\documentclass[12pt]{article}
\usepackage[top=2.5cm,left=2cm,right=2cm]{geometry}

\usepackage{hyperref}
\title{%Dirichlet Problem for Schr\"odinger Equations with Carleson measure Coefficinets and Reverse H\"older Potentials
Solvability of boundary value problem for Schr\"odinger Equations with Reverse H\"older Potentials on  $L^p$ and endpoint spaces }
\author{Xiao Botian, Tang Lin}
\date{}

\usepackage{appendix}
\usepackage{amsfonts}
\usepackage{amsmath}
\usepackage{amsthm}
\usepackage{amssymb}
\usepackage{esint}
\newtheorem{definition}{Definition}[section]
\newtheorem{theorem}{Theorem}[section]
\newtheorem{lemma}[theorem]{Lemma}

\newtheorem{remark}{Remark}[section]

\newcommand{\reals}{\mathbb{R}}

\newcommand{\rn}{\mathbb{R}^n}
\newcommand{\rr}{\mathbb{R}^{n+1}}

\newcommand{\rnp}{\mathbb{R}^{n+1}}

\newcommand{\lnorm}{\left\|}
\newcommand{\rnorm}{\right\|}
\newcommand{\labs}{\left|}
\newcommand{\rabs}{\right|}
\newcommand{\defeq}{\overset{\textup{def}}{=}}
\newcommand{\supp}{\textup{supp}\,}

\newcommand{\brackets}[1]{\left(#1\right)}
\newcommand{\absvalue}[1]{\left|#1\right|}
\newcommand{\rref}[1]{(\ref{#1})}

\newcommand{\nbltg}{\nabla_{\tan}}
\newcommand{\DIV}{\,\textup{div}\,}
\newcommand{\tilN}{\tilde{\mathcal{N}}}
\newcommand{\calS}{\mathcal{S}}
\newcommand{\calN}{\mathcal{N}}
\newcommand{\average}{\textup{Avg}\,}
\newcommand{\Avg}[1]{\underset{#1}{\average}}
\begin{document}
\allowdisplaybreaks[4]
\maketitle
\begin{abstract}
    %In this article we discuss the solvability on $\Omega$ a domain above a Lipschitz graph in $\rnp$, of the Dirichlet boundary value problem  of elliptic Schr\"odinger-type equation $-\DIV(A(x,t)\nabla u(x,t))+V(x,t)u(x,t)=0$ with boundary data $\restrict{u}_{\partial\Omega}=f\in L^p$ for $p>1$, with the assumption on the coefficient matrix $A$ that $|\nabla A|^2 t$ is a Carleson measure on $\omega$ and $V\in \mathcal{B}_q$ a non-negative funuction of Reverse H\"older class with $q>\frac{n}{2}$. We consider the $L^p$ norm equivalence of the $p$-adapted square function $\calS_p(u)$ and the nontangential maximal function $\tilN(u)$, and derive the $L^p$ norm controlling relation of $\tilN(u)$ by the data $f$ via a similar inequality.
    In this paper we discuss the solvability of the Neumann and Regularity boundary value problem  of elliptic Schr\"odinger-type equation $-\DIV(A(x)\nabla u(x,t))+V(x)u(x,t)=0$ with bounded measurable uniformly elliptic coefficinets $A(x)$ independent of $t$ and $V$ in Reverse H\"older class $\mathcal{B}_q$, and Neumann boundary data $\partial_{\nu_A}u(x,0)=f(x)\in  H^p_{\mathcal{L}}(\rn)$, or Regularity data $u(x,0)=g\in H^{1,p}_V(\rn)$, utilizing the method of layer potential. The solvability is proved when $A$ is a small $L^\infty$ perturbation of a matrix satisfying De Giorgi-Nash-Moser bounds. For $1-\varepsilon^\prime<p\le 1$ we give a new molecular decomposition of $H^p_\mathcal{L}$ and proved the solvability of Neumann and Regularity problem by establishing the Rellich inequalities on the boundary. Besides we also give the Campanato norm estimate of the double layer potential related to the Dirichlet problem with boundary data in certain Campanato-type spaces.
\end{abstract}

\section{Introduction}
Let $n\ge 3$, and write $(x,t)\in\rn\times \reals$ for an arbitrary point in the space. Consider the Schr\"odinger equation in the following form:
\begin{equation}\label{mainequation}
    \mathcal{L}u=-\DIV(A(x)\nabla u(x,t))+V(x)u(x,t)=0
\end{equation}
with $A(x)\in M^{(n+1)\times (n+1)}(\rn)$ being a bounded measurable coefficient matrix satisfying uniform elliptic condition 
\begin{align*}
    &\frac{1}{\Lambda}|\xi|^2\le \Re(\xi^\dag A\xi), &|\xi^\dag A \eta|\le \Lambda|\xi||\eta|,
\end{align*}
and $V(x)\ge 0$ being the potential function, both independent of $t$. Moreover we assume that $V\in \mathcal{B}_q(\rn)$ being a reverse H\"older potential with $q>\frac{n}{2}$. Such type of equation is a generalization of the classical Schr\"odinger equation $-\Delta u+Vu=0$, whose boundary value problems were widely studied in \cite{NeumannSchShen,NeumannSchTao,RegSchrodingerTao}. The equation \rref{mainequation} can arise from several different problems, such as solving the classical Schr\"odinger equation on a domain with Lipschitz boundary. 

 Studies of boundary value problems associated with \rref{mainequation} mainly focus on Neumann problems $(N)_p $ and Regularity problems $(R)_p$, as Dirichlet problems for positive potential are relatively easy to be solved by comparison with the elliptic equation $-\DIV(A\nabla u)=0$. The problems are formulated by
 \begin{equation*}
     \begin{array}{cc}
         (N)_p\left\{\begin{aligned}
             &-\DIV(A(x)\nabla u)+V(x)u=0,\\&\partial_\nu u(\cdot,0)=g\in H_{\mathcal{L}}^p,\\&\left\|\tilN(\nabla_V u)\right\|_p<\infty,
         \end{aligned}\right. &  (R)_p\left\{\begin{aligned}
             &-\DIV(A(x)\nabla u)+V(x)u=0,\\&u(\cdot,0)=F\in H_V^{1,p},\\&\left\|\tilN(\nabla_V u)\right\|_p<\infty,
         \end{aligned}\right.
     \end{array}
 \end{equation*} 
 where $\nu(x) =-\mathbf{e}_{n+1}\cdot A(x) $ is the outer normal vector associated with matrix $A$ and $\nu^\ast$ with $A^\ast$, the conjugate matrix of $A$. Note that the conormal derivative $\partial_\nu=\partial_\nu^+$ and $\partial_\nu^-$ are defined in the variational sense, namely
 \begin{align*}
            \int_{\rnp_\pm}(A\nabla u\nabla\Phi+Vu\Phi)dxdt=\int_{\rn}\partial_\nu^\pm u\cdot\varphi dx
 \end{align*} for every solution $u$ to \rref{mainequation} and $\varphi\in C_c^\infty(\rn)$ and $\Phi(x,t)\in C_c^\infty(\rnp)$ such that $\Phi(x,0)=\varphi(x)$. 
 $\nabla _Vu=\left(\nabla u,V^\frac{1}{2}u\right)$ denotes the total derivative associated with Schr\"odinger operator $\mathcal{L}$. The nontangential maximal function is defined by taking average on Whitney cubes, as
 \begin{equation*}
     \tilN(w)(x)=\sup_{\left|x^\prime -x\right|<\gamma t^\prime  }\brackets{\fint_{Q_{\theta t^\prime}\left(x^\prime,t^\prime\right)} |w(\xi,\tau)|^2 d\xi d\tau}^\frac{1}{2}
 \end{equation*}
where $\gamma>0$ and $0<\theta<1$ are suitable fixed parameters, and we take $\gamma=1$ and $\theta=\frac{1}{4}$ throughout the paper for convenience. The spaces $H_{\mathcal{L}}^p$ and $H^{1,p}_V$ are defined in Section \ref{endpointspace}. The boundary value is interpreted as convergence non-tangentially when $p\ge 1$, and convergence in the distribution sense when $p<1$.

 Given the existence of fundamental solution $\Gamma(x,t|\xi ,\tau)$, layer potentials then become a useful tool to formulate the solution of \rref{mainequation}. As is widely known, the single layer potential 
 \begin{equation*}
     \calS f(x,t)=\int_{\rn}\Gamma(x,t|\xi,0)f(\xi)d\xi
 \end{equation*}
and the double layer potential $(t\ne 0)$
\begin{equation*}
    \mathcal{D} f(x,t)=\int_{\rn}\overline{\partial_{\nu^\ast,\xi}\Gamma^\ast(\xi,0|,x,t)}f(\xi)d\xi
\end{equation*}
are both solutions on both the upper and lower half space, where $\partial_{\nu^\ast}$ denotes the adjoint exterior conormal derivative.

As is customary, we then define the single- and double-layer potential operators, associated $\mathcal{L}$, in the upper and lower half-spaces, $\rr_+$ and $\rr_-$, respectively, by
 \begin{align*}
     \calS^\pm f(x,t)=&\int_{\rn}\Gamma(x,t|\xi,0)f(\xi)d\xi,\quad (x,t)\in\rr_\pm,\\
    \mathcal{D}^\pm f(x,t)=&\int_{\rn}\overline{\partial_{\nu^\ast,\xi}^\pm\Gamma^\ast(\xi,0|x,t)}f(\xi)d\xi,\quad (x,t)\in\rr_\pm,
\end{align*}
and
$$\tilde{\mathcal{K}}f=-\frac I 2+\partial_{\nu}\calS^+f(\cdot,0)=\frac I 2-\partial_{\nu}^-\calS^-f(\cdot,0).$$

For elliptic equation with non-smooth coefficients satisfying the De Giorgi-Nash-Moser local boundedness property, layer potential method produces positive results when the coefficient is in real symmetric, up to a small perturbation.

 By block form we refer to matrices like \begin{equation*}
     A(x)=\begin{pmatrix}
         A_{\parallel\parallel}(x) & 0 \\
         0 & A_{\perp \perp}(x)
     \end{pmatrix}
 \end{equation*}
 where  $A_{\parallel\parallel}(x)\in M^{(n-1)\times (n-1)}(\reals)$, equivalently saying, the term $\partial^2_t u$  can be separated from other second order terms. $L^2$ Schr\"odinger boundary problems with this kind of coefficients and $A_{\perp\perp}=1$ can be solved by referring to the Kato problem, see Bailey \cite{BaileyPotential}. For block or complex Hermitian coefficients, $L^2$ boundary value problem is solved by first-order method, provided by Morris-Turner \cite{MorrisTurner}, where small $L^\infty$ perturbations of second-order coefficients were studied as well; see also \cite{CriticalPerturbationHofmann1,CriticalPerturbationHofmann2}. 

 Research on $L^p$ boundary value problems of second-order linear elliptic equations with non-smooth coefficients can be dated back to Dahlberg-Kenig \cite{DahlbergKenigLaplace}, where $L^p$ and $H^p$ Neumann and Regularity problem of Laplace equation on Lipschitz domains were proved to be well-posed. Then in Kenig-Pipher \cite{KenigPipher} Neumann and Regularity problems on unit ball were considered, and their arguments were based on estimates on the harmonic measure and Green's function. Alfonseca-Angles-Axselsson-Hofmann-Kim \cite{AAAHK} used the layer potential method to solve $L^2$ Neumann and Regularity problem, with the assumption that $A(x)$ is independent of $t$ with $A$ being complex constant, real symmetric or in the block form. With the help of layer potential, the solvability is converted into the invertibility of boundary layer potential operators as a type of singular integral operators, which can be estimated by analyzing their kernel functions. In \cite{AAAHK}, perturbations of coefficients were also discussed, and the $L^2$ boundedness of square function estimate were proved to be stable under  $L^\infty$ perturbation. Hofmann-Mitrea-Morris \cite{MethodOfLayerPotential} discussed the $H^p$ estimations of layer potential operators and solved the $L^p$ Neumann and Regularity problem within $1-\varepsilon^\prime<p<2+\varepsilon$ and Dirichlet problem in $\Lambda^\beta$ within $0\le \beta <\alpha$ for some small $\varepsilon, \varepsilon^\prime $ and $\alpha$. 
 
 Other methods such as first order method are mainly used for the study of elliptic systems, with examples such as \cite{LayerPotentialsBeyond}. Their arguments can also be generalized for $L^p$ situation, see Auscher-Mourgoglou \cite{RepresentationFirstOrder} and Auscher-Stahlhut \cite{FunctionalFirstOrderDirac}. These methods influenced \cite{MorrisTurner}, where the authors constructed a Dirac-type operator $DB$ associated with $\mathcal{L}$ that admits a bounded functional calculus. Ros\'en \cite{LayerPotentialsBeyond} also provided a different perspective to study layer potential operators, following which we can get 
 \begin{align*}
     \hat{A}\nabla_V\calS f=&\brackets{ e^{-tDB}\chi_{+}(DB)\begin{pmatrix}
          f  \\
          0 
     \end{pmatrix}}_\perp,\\
     \mathcal{D} f=&\brackets{ e^{-tBD}\chi_{+}(BD)\begin{pmatrix}
          f  \\
          0 
     \end{pmatrix}}_\perp,
 \end{align*}
 as a corollary we have
 \begin{equation}
 \sup_{t>0}\brackets{\|\nabla_V\calS f(\cdot,t)\|_2+\|\mathcal{D}f(\cdot,t)\|_2}\le C\|f\|_2.\label{L^2 boundedness of slices} 
 \end{equation}
 Several other authors such as Borts-Hofmann-Luna Garcia-Mayboroda-Poggi \cite{CriticalPerturbationHofmann1,CriticalPerturbationHofmann2} also presented $L^p \,({2-\varepsilon^\prime<p<2+\varepsilon})$ well-posedness of boundary value problems for equations with lower-order coefficients in $L^{n}$ and $L^\frac{n}{2}$ via abstract layer potential operators and square function estimates.

For Schr\"odinger equation of non-smooth coefficients, it still remains a question: Is it possible to give its layer potential solution to $(N)_p$ and $(R)_p$,  what conditions should the coefficient $A$ satisfy, and in what range should the index $p$ take? 

To handle the solvability when $1-\varepsilon^\prime<p<2$, we need to set up the boundedness and invertibility for the boundary layer potential operators at endpoint spaces $H^1_\mathcal{L}$. Fortunately, boundedness on $H^1_\mathcal{L}$ of various singular integral operators derived from Classical Schr\"odinger operator $-\Delta+V$, such as Riesz transform and $(-\Delta+V)^{it}$, were discussed widely in \cite{DongHardy,CZParabolicSchrodingerTrong,NonCancelledAtom}, and it is new and very interesting for us to investigate the boundedness of boundary layer operators associated with the generalized Schr\"odinger operator with variable coefficient $-\DIV(A\nabla)+V$, which lacks smoothness in the $x$ direction, especially on the space $H^p_\mathcal{L}$ when $\frac{n}{n+1}<p<1$.

A great difference between Schr\"odinger equation \rref{mainequation} and elliptic equation without lower order terms $-\DIV(A \nabla u)=0$, is that $u-a$ is no longer a solution to the equation if the constant $a\ne 0$. Nevertheless, the usage of local $L^p$ Fefferman-Phong inequality given in \cite{BenAli,MorrisTurner} may assist us to overcome this difficulty.

Now we can introduce the main result of this paper.

\begin{theorem}\label{mainthm}
    Given $A_0(x)$ a uniformly elliptic matrix and the potential $V\in \mathcal{B}_q(\rn)$ with  $q\ge\frac {n+1}2$. 
    
    Suppose $A_0$ is real symmetric, or $A_0$ is of  the block form and any all the weak solutions to both $\DIV(A_0\nabla u)=0$ and $\DIV(A_0^\ast\nabla u)=0$ possess De Giorgi-Nash-Moser property(see\rref{boundednesselliptic} and \rref{holderelliptic}).
    
    Then there exists an $\epsilon_0>0$ sufficiently small depending on dimension, ellipticity constant, DG-N-M constants and $[\![V]\!]_q$, such that if the coefficient $A(x)=A_0(x)+\epsilon_0 \tilde{A}(x)$ where $ \lnorm \tilde{A}(x)\rnorm_\infty \le 1$, and there exists $0<\varepsilon^\prime <\frac{1}{n+1}$ depending on dimension, ellipticity constant and $[\![V]\!]_q$ such that $\left(N\right)_p $ and $(R)_p$ are respectively uniquely solvable by layer potential in $H^p_\mathcal{L}$ and $H^{1,p}_V$ for $1-\varepsilon^\prime <p\le 2$.
\end{theorem}

%\ifdef

%As a special case of Theorem \ref{mainthm}, we have the following corollary for small complex perturbation near real coefficients.

%\begin{corol}
%and the potential $V\in \mathcal{B}_q(\rn)$ with  $q\ge\frac {n+1}2$.  
    
  %  Then there exists an $\epsilon_0>0$ sufficiently small depending on dimension,  ellipticity constant and $[\![V]\!]_q$, such that if the coefficient $A(x)=A_0(x)+\epsilon_0 \tilde{A}(x)$ where $\tilde{A}$ is a complex matrix satisfying $ \lnorm \tilde{A}(x)\rnorm_\infty \le 1$, and $0<\varepsilon^\prime <\frac{1}{n+1}$ depending on dimension, ellipticity constant and $[\![V]\!]_q$ such that $\left(N\right)_p $ and $(R)_p$ are respectively uniquely solvable by layer potential in $H^p_{\mathcal{L}}$ and $H^{1,p}_V$ for $1-\varepsilon^\prime <p\le 2$.
%\end{corol}
%\fi

Moreover, by substitution of variables, the above conclusions can be generalized to domains above a Lipschitz graph $\Omega=\{(x,t)\mid t>\phi(x)\}$ where $|\phi(x)-\phi(y)|\le \lambda|x-y|$, as long as the parameters $\gamma$ and $\theta$ for the maximal function are proper.

The organization of this paper is as follows. In Section \ref{section:preliminaries} we give the bounds for the fundamental solution $\Gamma(x,t|\xi,\tau)$ and the integral of its spatial derivative on dyadic annuli, and define the endpoint spaces  $H^p_{\mathcal{L}}$ and $\Lambda ^\beta_V$, as well as some properties of approximate molecules in $H^p_{\mathcal{L}}$.   Section \ref{section: estimates for layer potentials} is the main part of the article,  we give some $H^p_{\mathcal{L}}(1-\varepsilon<p\leq 1 )$ estimates of the single layer potential and its maximal function, as well as $L^p\ (2\le p< 2+\varepsilon)$ estimates for single layer potential and and $\Lambda^\beta_V$ properties for double layer potential respectively in Section \ref{layer potential when 2+epsilon} and \ref{double layer potential lambda}. In Section \ref{section: rellich estimates on the boundary} we make a further step towards the invertibility of layer potential operators at the boundary, where new Rellich type inequalities in Hardy spaces are established. Section \ref{section: solution of equation} makes the final conclusion about the existence of layer potential solution and the uniqueness of solutions with bounded maximal function norm, proving our main Theorem \ref{mainthm} and the convergence of solutions to their boundary value. Although in this paper we follow some ideas in \cite{MethodOfLayerPotential}, we need new techniques  to overcome the influences of the potential $V$.

Without special announcements, all the constants written by $C$, $c$ or other letters are dependent on the dimension $n$, and dependent on the ellipticity constant of matrix $A$ and the constant $[\![V]\!]_q$ of potential $V$ if the context refers to the equation. As the tradition shared by most papers on analysis, the constants denoted by $C$ or $c$ may take different values in different inequalities. Also we use subscripts on $C$ to indicate the parameters that the constant depends on.

\section{Preliminaries}\label{section:preliminaries}%Basic Definitions and Preliminaries}

\subsection{Consequences of Reverse H\"older Inequalities}

Let $|S|$ represent the Lebesgue measure of a point set $S$, and $\Avg{S}f=\fint_S f=\frac{1}{|S|}\int f$. Suppose $V(x)\in \mathcal{B}_q(\rn)$ with $q>1$ is a reverse H\"older potential, i.e. $V$ satisfies the inequality

\begin{equation*}
    \brackets{\fint_Q V(x)^q dx}^\frac{1}{q}\le C\fint_Q V(x) dx
\end{equation*}
for every cube $Q=Q_r(x)\subset \rn$ of radius $r$ and centered at $x$. The infimum for such constant $C$ is denoted by $[\![V]\!]_q$. By the property of reverse H\"older functions, we know $V^\delta$ are all reverse H\"older functions in $\mathcal{B}_{\frac{q}{\delta}}$ for $0<\delta<q$, with constants depending on $[\![V]\!]_q$ and $\delta$, and  if $V(x)\in \mathcal{B}_q(\rn)$ with $q>1$, then $V(x)\in \mathcal{B}_{q+\epsilon}(\rn)$  for some  $\epsilon>0$.

In the situation when $q\ge \frac{n}{2}$ we recall the well-known maximal function $m_V(x)$ in \cite{LpShen} defined by \begin{align*}
    m_V(x)=\inf\left\{\frac{1}{r}>0\Big| \int_{Q_r(x)} V(\xi)d\xi\le  r^{n-2}\right\},
\end{align*}and 
\begin{equation}\label{property of mv}
    \frac{c m_V(x)}{\left(1+|x-y|m_V(x)\right)^{\frac{\kappa_0}{\kappa_0+1}}}\le m_V(y)\le Cm_V(x)\left(1+|x-y|m_V(x)\right)^{\kappa_0},
\end{equation}
for some constant $\kappa_0>0 $,implying that $m(x)\sim m(y) $ for $x\sim y$. We also write $\rho(x)=\frac{1}{m_V(x)}$ for the critical radius function. By \cite{LpShen} we know that the function \begin{equation}
    r^2\fint_{Q_r(x)}V(\xi)d\xi\le C(rm_V(x))^{\delta_0}, \label{zeta}
\end{equation} if $r<C\rho(x)$, where $\delta_0 =2-\frac{n}{q}$.

As is known, the Fefferman-Phong inequality for $V\in \mathcal{B}_{q}(\rn)$ with $q\geq n/2$ on Lipschitz domain $\Omega$
\begin{equation}\label{globalfeffermanphong}
 \int_{\Omega} |F(x)|^2m_V(x)^2 dx \le C_{\Omega}\int_{\Omega}\brackets{|\nabla F|^2+V(x)|F(x)|^2}dx,
\end{equation}
where $C_\Omega$ means a constant dependent on the Lipschitz character of $\Omega$, and its local version for $1\leq p<\infty$  and $V^\frac{p}{2}\in \mathcal{B}_{q}(\rn)$ with $q>1$
 \begin{equation}\label{localfeffermanphong}
     \min\left\{\frac{1}{r^p},\fint_{Q_r(x_0)}V^\frac{p}{2}(\xi)d\xi \right\}\int_{Q_r(x_0)} |F(x)|^pdx\le C \int_{Q_r(x_0)} \brackets{|\nabla F|^p+V^\frac{p}{2}|F|^p}dx
 \end{equation}
 describe the relations between the $L^p$ norm of a $W^{1,p}_V$ function $F$ and its total derivative $\nabla_V F$. where the constant $C$ depends on $n, p$ and $q$.

Let  $q>s\ge 0, \alpha>0$,\ $q\ge \max\left\{1,\frac{ns}{\alpha}\right\}$, we also have the following integral estimate for some constant $N>0$ and $C$,
\begin{align}
    \int_{\rn}\frac{V(x)^s \, dx}{(1+m_V(x_0)|x-x_0|)^N|x-x_0|^{n-\alpha}}\le Cm_V(x_0)^{2s-\alpha}.\label{integrability of V}
\end{align}
This illustrates that $V^s$ can be regarde as a tempered distribution.

 %For $V\in \mathcal{B}_q$ and $q>\frac{n+1}{2}$, i
 It is well-known (see for example \cite{MoserHarnackElliptic}) that if the coefficient $A$ is real, then the weak solution $u$ of %$\mathcal{L}u=0$ or $\mathcal{L}^\ast u=0$ 
 $-\DIV(A^\ast\nabla w)=0$ satisfies De Giorgi-Nash-Moser local boundedness and local H\"older continuity with exponent $0<\alpha_0\le 1$ in a domain $\Omega\subset\rnp$  such that for any cube $Q=Q_r(x_0,t_0)\subset\rnp$, of radius $r$ and centered $(x_0,t_0)\in\rn\times\reals$, whose concentric $\sigma Q=Q_{\sigma r}(x_0,t_0)\subset\Omega\, (\sigma>1)$ as follows:
\begin{align}
        \sup_{Q_r(x_0,t_0)} |w| &\le C_\sigma \brackets{\fint_{Q_{\sigma r}(x_0,t_0)} |w|^2}^\frac{1}{2},\label{boundednesselliptic}\\
         |w(X)-w(Y)| &\le C_\sigma\frac{|X-Y|^{\alpha_0}}{r^{\alpha_0}} \brackets{\fint_{Q_{\sigma r}(x_0,t_0)} |w|^2}^\frac{1}{2},\label{holderelliptic}
    \end{align}
 wherever $X,Y\in Q_r(x_0,t_0)$.

 Considering the Schr\"odinger equation \rref{mainequation} with potential $V\in \mathcal{B}_q(\rn)$ and $q>\frac{n}{2}$, by the arguments in Section \ref{Appendix A}, if the coefficient $A$ is complex, and \rref{boundednesselliptic} and \rref{holderelliptic} hold for both $-\DIV(A\nabla w)=0$ and its adjoint equation $-\DIV(A^\ast\nabla w)=0$  in a domain $\Omega\subset\rnp$, then  the following DG-N-M properties also hold for weak solutions to $\mathcal{L}u=0$ and $\mathcal{L}^\ast u=0$ for some constant $N_0>0$, with a uniform H\"older exponent $0<\alpha\le\min\{\alpha_0,\delta_0\}$, such that $\sigma B=Q_{\sigma r}(x_0,t_0)\subset\Omega\,(\sigma>1)$, 
    \begin{align}
        \sup_{Q_r(x_0,t_0)} |u| &\le C_{\sigma} \brackets{\fint_{Q_{\sigma r}(x_0,t_0)} |u|^2}^\frac{1}{2}\brackets{1+r^2\fint _{Q_{\sigma r}(x_0)}Vd\xi}^{N_0},\label{degiorginash}\\
         |u(X)-u(Y)| &\le C_\sigma\frac{|X-Y|^\alpha}{r^\alpha} \brackets{\fint_{Q_{\sigma r}(x_0,t_0)} |u|^2}^\frac{1}{2}\brackets{1+r^2\fint _{Q_{\sigma r}(x_0)}Vd\xi}^{N_0+1},\label{moser}
    \end{align}
 wherever $X,Y\in Q_r(x_0,t_0)$. Note that in \cite{JiangLiSchrodinger} a proof is provided for real symmetric coefficients.
        
Furthermore we can combine \rref{property of mv}, \rref{globalfeffermanphong} and \rref{degiorginash}, \rref{moser} to get 
\begin{align}
    \sup_{Q_r(x_0,t_0)} |u| \le& \frac{C_N}{(1+rm_V(x_0) )^{N}} \brackets{\fint_{ Q_{\sigma r}(x_0,t_0)} |u|^2}^\frac{1}{2}\label{boundedness with decay}\\
    \sup_{X,Y\in Q_r(x_0,t_0)} |u(X)-u(Y)| \le& \frac{C_N|X-Y|^\alpha}{(1+rm_V(x_0) )^{N}r^\alpha} \brackets{\fint_{Q_{\sigma r}(x_0,t_0)} |u|^2}^\frac{1}{2}\label{holder with decay}
\end{align}
for every $N\ge 0$. 

In the following discussions we always assume $V\in\mathcal{B}_q(\rn)$ where $q \ge\max\{2,n/2\}$.  We also make the assumption that any weak solution of $-\DIV(A\nabla u)=0$ and its conjugate equation $-\DIV(A^\ast\nabla u)=0$ always satisfy DG-N-M properties \rref{boundednesselliptic} and \rref{holderelliptic}.  Then by Section \ref{Appendix A}, we know that \rref{boundedness with decay} and \rref{holder with decay}  hold for weak solutions to $\mathcal{L}u=0$ and $\mathcal{L}^\ast u=0$. Therefore, from \cite{LayerPotentialsBeyond,MorrisTurner} we know that the fundamental solutions $\Gamma(x,t|\xi,\tau)$ and $\overline{\Gamma(\xi,t|x,\tau)}$ of  $\mathcal{L}$ and $\mathcal{L}^\ast$ satisfy

 \begin{align} 
         |\Gamma(x,t|\xi,\tau)|+|\Gamma(\xi,\tau|x,t)|&\le \frac{C_N}{(1+(|x-\xi|+|t-\tau|)m_V(x))^N (|x-\xi|+|t-\tau|)^{n-1}}.
\end{align}

\subsection{Estimates on Fundamental Solutions and Green Functions}\label{estimates fundamental solutions green}

First we need a Caccioppoli-type inequality to bound the $L^2$ average of total derivative of a solution on hyperplane cubes parallel to $\rn\times 0$.
\begin{lemma}\label{basiclemma}There exists an $\varepsilon>0$ such that for every $0\le \tilde{\varepsilon}<\varepsilon$, if $\mathcal{L}u=0$ on $Q_{2r}(x)\times(t-r,t+r)$, we have
    \begin{align*}
        \brackets{\fint_{Q_r(x)}\labs \nabla_V u(\xi,t)\rabs ^{2+\tilde{\varepsilon}} d\xi}^\frac{1}{2+\tilde{\varepsilon}}\le C\brackets{\fint_{t-2r}^{t+2r}\fint_{Q_{2r}}|\nabla_V u|^2d\xi d\tau}^\frac{1}{2},\\
        \brackets{\fint_{Q_r(x)}\labs \nabla_V u(\xi,t)\rabs ^{2+\tilde{\varepsilon}} d\xi}^\frac{1}{2+\tilde{\varepsilon}}\le \frac{C}{r}\brackets{\fint_{t-2r}^{t+2r}\fint_{Q_{2r}}|u|^2d\xi d\tau}^\frac{1}{2}.
    \end{align*}
\end{lemma}

The proof of this lemma is basically the same as in \cite[Prop. 2.1]{AAAHK} since $V(x)$ is independent of $t$, or we refer the readers to the proof of Lemma \ref{differencefundamentalsolution} below.

Write $\Theta_r(x,t)=Q_{2r}(x,t)\setminus Q_r(x,t)$ the dyadic annulus with radius from $r$ to $2r$ and $\tilde{\Theta}_{r}(x,t)=Q_{4r}(x,t)\setminus Q_{\frac{r}{2}}(x,t)$. For the fundamental solution $\Gamma(x,t|\xi,\tau)$ we have the following estimates derived from Lemma \ref{basiclemma} by elementary calculations.
\begin{lemma} For any $N>0$ and $m\geq 0$, if $ \gamma_1\left(|x-x^\prime|+|t-t^\prime|\right)\le |x-\xi|+|t-\tau| \quad (\gamma_1>1)$, we then have
    \begin{align} 
                &|\partial_t^m\Gamma(x,t|\xi,\tau)|\le \frac{C_N}{(1+(|x-\xi|+|t-\tau|)m_V(x))^N (|x-\xi|+|t-\tau|)^{n-1+m}},\\
        &\labs\partial_t^m\brackets{\Gamma(x,t|\xi,\tau)-\Gamma(x^\prime,t^\prime|\xi,\tau)}\rabs\\&\le \frac{C_N (|x-x^\prime|+|t-t^\prime|)^\alpha}{(1+(|x-\xi|+|t-\tau|)m_V(x))^N(|x-\xi|+|t-\tau|)^{n-1+m+\alpha}},\\
        &\int_{\Theta_{2^kr}(x)}|\nabla_V\Gamma(x,t|\xi,\tau)|^2d\xi\le \frac{C_N2^{-kn}r^{-n}}{(1+2^krm_V(x))^N}, \\
        &\int_{\Theta_{2^kr}(x)}\labs \nabla _V \brackets{\Gamma(x^\prime,t^\prime|\xi,\tau)-\Gamma(x,t|\xi,\tau)}\rabs ^2 d\xi
        \le \frac{C_N 2^{-2k\alpha-kn}r^{-n}}{(1+2^krm_V(x))^N}.\label{gradientdifferenceholder}
    \end{align}
\end{lemma}

Note that $1+(|x-\xi|+|t-\tau|)m_V(\xi) $ in the denominators on the right hand side are interchangeable with $1+(|x-\xi|+|t-\tau|)m_V(x) $ due to the property \rref{property of mv}. 

Denote $\Gamma_0(x,t|\xi,\tau)$ be the fundamental solution of the elliptic equation $-\DIV(A\nabla u)=0$, then we can estimate the difference between these two fundamental solutions with the help of Lemma \ref{basiclemma}.

\begin{lemma}\label{differencefundamentalsolution}
Let  $\delta_0 =2-\frac{n}{q}$ as in \rref{zeta}.
For $|x-\xi|+|t-\tau|<\rho(x)$ and $k\ge 0$ we have
  \begin{align}
      &\labs\Gamma(x,t|\xi,\tau)-\Gamma_0(x,t|\xi,\tau)\rabs\le \frac{Cm_V(x)^{\delta_0}  }{(|x-\xi|+|t-\tau|)^{n-1-\delta_0 }},\label{gamma - gamma0 estimate}\\
      &\int_{\Theta_{2^{-k}r}(x)}\labs \nabla _V \brackets{\Gamma(x,t|\xi,\tau)-\Gamma_0(x,t|\xi,\tau)}\rabs ^2 d\xi
        \le C 2^{-2k\delta_0 +kn}r^{2\delta_0 -n}m_V(x)^{2\delta_0} ,\label{difference gradient fundamental solution}\\
        &\labs\brackets{\Gamma(x,t|\xi,\tau)-\Gamma_0(x,t|\xi,\tau)}-\brackets{\Gamma(x^\prime,t^\prime|\xi,\tau)-\Gamma_0(x^\prime,t^\prime|\xi,\tau)} \rabs\\&\le\frac{Cm_V(x)^{\delta_0} \brackets{|x-x^\prime|+|t-t^\prime|}^\alpha }{(|x-\xi|+|t-\tau|)^{n-1-\delta_0 +\alpha}},\label{double difference fundamental solution} 
        \end{align} 
       and for $|x-x^\prime|+|t-t^\prime|\le r$ and $2^{k+1} r\le\rho(x)$ with  $k\ge 0$, we have
       \begin{align}
        &\int_{\Theta_{2^kr}(x)}\labs \nabla _V \brackets{\brackets{\Gamma(x,t|\xi,\tau)-\Gamma_0(x,t|\xi,\tau)}-\brackets{\Gamma(x^\prime,t^\prime|\xi,\tau)-\Gamma_0(x^\prime,t^\prime|\xi,\tau)}}\rabs ^2 d\xi\nonumber
        \\&\le C 2^{2k\delta_0 -2k\alpha-kn}r^{2\delta_0 -n}m_V(x)^{2\delta_0}.\label{double difference integral gradient fundamental solution}
  \end{align}   
  
\end{lemma}
\begin{proof}
    By a similar argument in \cite{NeumannSchShen} we can prove \rref{gamma - gamma0 estimate}, as well as
 \begin{align}\label{partial t gamma - gamma 0}
     \labs\partial_t\left(\Gamma(x,t|\xi,\tau)-\Gamma_0(x,t|\xi,\tau)\right)\rabs&\le \frac{Cm_V(x)^{\delta_0}  }{(|x-\xi|+|t-\tau|)^{n-\delta_0 }}.
 \end{align}
 for $|x-\xi|+|t-\tau|<\rho(x)$.
 
 We then proceed to prove \rref{difference gradient fundamental solution}. Let $u(\xi,\tau)=\Gamma(x,t|\xi,\tau)-\Gamma_0(x,t|\xi,\tau)$, then by basic calculation we know that $\mathcal{L}^\ast u=V(\xi)\Gamma_0(x,t|\xi,\tau), \, \mathcal{L}^\ast \partial_t u=V(\xi)\partial_t\Gamma_0(x,t|\xi,\tau)$, and by Fundamental Theorem of Calculus we have
 \begin{align*}
     &\int_{\Theta_{2^k r}(x)}|\nabla_V u(\xi,\tau_0)|^2dx\\ \le &\frac{C}{2^k r}\int_{\tau_0-2^{k} r}^{\tau_0+2^{k} r}\int_{\Theta_{2^k r}(x)}|\nabla_Vu(\xi,\tau)|^2 d\xi d\tau+C2^k r\int_{\tau_0-2^{k} r}^{\tau_0+2^{k} r}\int_{\Theta_{2^k r}}|\nabla_V \partial_\tau u(\xi,\tau)|^2 d\xi d\tau \\
     \le &\frac{C}{2^{3k}r^3}\int_{\tau_0-2^{k+1} r}^{\tau_0+2^{k+1} r}\int_{\tilde{\Theta}_{2^k r}(x)}|u(\xi,\tau)|^2 d\xi d\tau+\frac{C}{2^{k}r}\int_{\tau_0-2^{k+1} r}^{\tau_0+2^{k+1} r}\int_{\tilde{\Theta}_{2^k r}}|\partial_\tau u(\xi,\tau)|^2 d\xi d\tau \\
     &+{C}{2^{k}r}\int_{\tau_0-2^{k+1} r}^{\tau_0+2^{k+1} r}\int_{\tilde{\Theta}_{2^k r}(x)}|V(\xi)\Gamma_0(x,t|\xi,\tau)|^2 d\xi d\tau\\&+{C}2^{3k} r^3\int_{\tau_0-2^{k+1} r}^{\tau_0+2^{k+1} r}\int_{\tilde{\Theta}_{2^{k }r}(x)}|V(\xi)\partial_t\Gamma_0(x,t|\xi,\tau)|^2 d\xi d\tau 
 \end{align*}
 by referring to Caccioppoli's inequality. Using \rref{gamma - gamma0 estimate}, \rref{partial t gamma - gamma 0} and the reverse H\"older inequality for $V$ we have
 \begin{align*}
     &\int_{\Theta_{2^k r}(x)}|\nabla_V u(\xi,\tau_0)|^2dx\\
     \le & C\frac{m_V(x)^{2\delta_0}(2^k r)^{n+1}}{(2^k r)^{2n+1-2\delta_0}}+C\frac{(2^k r)^{n+1}}{(2^k r)^{2n-3}}\left(\fint_{Q_{2^k r}} Vd\xi \right)^2\\
     \le& Cm_V(x)^{2\delta_0}(2^k r)^{-n+2\delta_0}+C(2^k r)^{-n+4}(2^k r)^{-4}(2^{-k} r^{-1}m_V(x))^{2\delta_0}\\
     =&Cm_V(x)^{2\delta_0} 2^{-nk+2k\delta_0}r^{-n+2\delta_0}.
 \end{align*}
 
 The estimate \rref{double difference fundamental solution} is a consequence of \rref{holder with decay} and \rref{gamma - gamma0 estimate}, and the deduction of \rref{double difference integral gradient fundamental solution} is similar to \rref{difference gradient fundamental solution}.
\end{proof}

By the solvability of $(N_2)$ and $(R_2)$ and De Giorgi-Nash-Moser property of solutions, Green functions and Neumann functions can be constructed, and it can be verified using arguments resembling \cite{GlobalGreenFunctionKimKang} that the following pointwise estimates hold, by referring to local boundedness and H\"older continuity near the boundary.

\begin{lemma}%\cite{}
    Suppose $A_0$ is of block form or real symmetric and $A$ satisfies the assumptions above for $\epsilon_0>0$ sufficiently small dependent on elliptic constants, DG-N-M constants and $[\![V]\!]_q$. Then the Neumann function and Green function on $\rnp_+$ satisfy the following estimates, 
    \begin{align}
         |N(x,t|\xi,\tau)|+|G(x,t|\xi,\tau)|&\le \frac{C_N}{(1+(|x-\xi|+|t-\tau|)m_V(\xi))^N(|x-\xi|+|t-\tau|)^{n-1}}\label{estiamteGreenNeumann}\\
        |N(x,t|\xi,\tau)-N(x,t|\xi_0,\tau_0)|&+|G(x,t|\xi,\tau)-G(x,t|\xi_0,\tau_0)|\nonumber\\&\le \frac{C_N (|\xi-\xi_0|+|\tau-\tau_0|)^\alpha}{(1+(|x-\xi|+|t-\tau|)m_V(\xi))^N(|x-\xi|+|t-\tau|)^{n-1+\alpha}}\label{differenceGreenNeumann}
    \end{align}
    for $ \gamma_1(|\xi-\xi_0|-|\tau-\tau_0|)\le (|x-\xi_0|+|t-\tau_0|) \quad (\gamma_1>1)$.
\end{lemma}

All the above estimates are applicable for the conjugate operator $\mathcal{L}^\ast$, hence they are still true with the two arguments of $\Gamma ,N$ or $G$ interchanged.

\subsection{Function Spaces}\label{endpointspace}

For $p\le 1$, Hardy space is a reasonable alternative for the $L^p$ space in the sense of interpolation. In the situation where Schr\"odinger equations are considered,  we must modify the space such that the integrals of its functions do not vanish. As in \cite{Dziubanski,YangDachunHardySpace,BuiDuongHardySpace}, for $\frac{n}{n+1}<p\le 1$ we define a $2$-atom for $H^p_\mathcal{L}(\rn)$ space as a function $a(x)$ satisfying the following three assumptions:
\begin{itemize}
    \item[(i)] $\supp a(x) \subset Q_r(x_0)$ with $r\le \gamma_0\rho(x_0)$ for some $\gamma_0>1$,
    \item[(ii)] $\|a\|_2\le C r^{\frac{n}{2}-\frac{n}{p}} $,
    \item[(iii)] $ \int_{\rn} a(x)dx=0$ for $r\le \rho(x_0)$.
\end{itemize}
for some positive constant $C$. The atomic local Hardy space $H^p_{\mathcal{L},\textup{at}}(\rn)$ consists of sums of atoms with certain coefficients, taken in the sense of distribution on $\mathcal{D}=C_c^\infty(\rn)$, given by
\begin{equation*}
    H^p_{\mathcal{L},\textup{at}}(\rn)=\left\{f=\sum_{k=1}^{\infty}\lambda_ka_k(x)\in\mathcal{D}^\prime(\rn)\mid\|f\|_{H^p_\mathcal{L},\textup{at}}^p\defeq\sum_{k=1}^{\infty}|\lambda_k|^p<\infty\right\}.
\end{equation*}

Also the $H^p_\mathcal{L}$ space defined by $\rho$-local maximal function
\begin{align}
        \mathcal{P}_\rho f(x)=\sup_{0<t<\rho(x)}\absvalue{\int_{Q_r(x_0)}\varphi_t(x-\xi)f(\xi)d\xi}\label{definitionofPV}
    \end{align}
    where $\varphi\in C_c^\infty(\rn)$ with $\int_{\rn}\varphi=1$, and $H^p_\mathcal{L}=\left\{f\Big|\|f\|_{H^p_\mathcal{L}}\defeq \|\mathcal{P}_{\rho} f\|_{p}<\infty\right\}$. In fact $H^p_{\mathcal{L},\textup{at}}=H^p_{\mathcal{L}}$.
    
    For the space $H^p_\mathcal{L}$ we would mention a special case when $V\equiv 0$, the function $\rho(x)\equiv +\infty$, then the space degenerates to the usual Hardy space $H^p_{\Delta}=H^p$, as stated in \cite{FeffermanStein}. Also we note that for $1<p<\infty$ we have $H^p_\mathcal{L}(\rn)=H^p(\rn)=L^p(\rn)$ since the maximal function $\mathcal{P}_\rho f$ is controlled by $\mathcal{M}f$ and $|f|$ is controlled by $\mathcal{P}_\rho f$.

Now we introduce the approximate molecule definition and prove its properties.

\begin{definition}
    An approximate molecule of $H^p_\mathcal{L}$ associated with a cube $Q_r(x_0)$ is defined by the function $m(x)$ satisfying the following four assumptions.
    
    \begin{itemize}
        \item[(i)] $r\le \gamma_0\rho(x_0)$ for some $\gamma_0 >1$, 
        \item[(ii)] $\lnorm m\chi_{Q_r}\rnorm_2\le C r^{\frac{n}{2}-\frac{n}{p}} $,
        \item[(iii)] $\lnorm m\chi_{\Theta_{2^k r}}\rnorm_2\le C 2^{-\delta k} (2^kr)^{\frac{n}{2}-\frac{n}{p}} $ for some $\delta >0$,
        \item[(iv)] \begin{equation*}
            \labs\int_{\rn} m(x)dx\rabs\le\left\{
            \begin{aligned}
                &C\left(\ln\left(1+\frac{r}{\rho(x_0)}\right)\right)^{-1}, &p=1,\\
                &C\rho(x_0)^{n-\frac{n}{p}},&\frac{n}{n+1}<p<1.
            \end{aligned} \right.
        \end{equation*} 
    \end{itemize}
\end{definition}
Note that the case $p=1$ is already defined in \cite{NonCancelledAtom}.

As expected, any approximate molecule $m$ belongs to the space $H^p_\mathcal{L}$ and $\|m\|_{H^p_\mathcal{L}}\le C$ with $C$ uniform in $x_0$ and $r$. The case that $p=1$ is essentially given in \cite{NonCancelledAtom}, for completeness, we will prove the case when $\frac{n}{n+1}<p\le 1$ following the thoughts in \cite{InhomogeneousCancellation}.

\begin{lemma}\label{noncancellingatom}
    Let $\frac{n}{n+1}<p\leq 1$. Let $\tilde{a}(x) $ be a function supported in $Q_r(x_0)$ with $r\le \gamma_0\rho(x_0)$ $(\gamma_0>1)$ and $\|\tilde{a}\|_2\le Cr^{\frac{n}{2}-\frac{n}{p}}$,   together with the ``approximate cancellation'' condition $$\absvalue{ \int_{\rn} \tilde{a} dx}\le C\rho(x_0)^{n-\frac{n}{p}}.$$ Then we have $\|\tilde{a}\|_{H^p_\mathcal{L}}\le C$ uniformly.
\end{lemma}
\begin{proof}
    Let $\varphi(x)\in C_c^\infty(\rn)$ be a bump function satisfying $\varphi(x)\ge 0, \int_{\rn}\varphi(x)dx=1$ and $\varphi(0)=1$, and $\varphi_t(x)=\frac{1}{t^n}\varphi\brackets{\frac{x}{t}}$ be its dilation. Then $\varphi_t(x)\le \frac{C_N}{t^n\brackets{1+\frac{|x|}{t}}^N}$ for every $N>0$.  For the $L^p$ norm of the $\rho$-local maximal function $\mathcal{P}_\rho \tilde{a}(x)$, we decompose its integral into two parts.
    \begin{align*}
        \int_{2Q_r(x_0)}|\mathcal{P}_\rho\tilde{a}(x)|^p dx\le \int_{2Q_r(x_0)}|\mathcal{M}\tilde{a}(x)|^2 dx\le Cr^{n-\frac{np}{2}}\|a\|_2^p\le C.
    \end{align*}
    Also,% let $R(x)=\max\{r,\rho(x)\}$ we have
    \begin{align*}
        &\int_{\brackets{2Q_r(x_0)}^\complement}|\mathcal{P}_\rho\tilde{a}(x)|^p dx\\\le &\int_{\brackets{2Q_r(x_0)}^\complement}\sup_{0<t<\rho(x)}\absvalue{\int_{Q_r(x_0)}\varphi_t(x-\xi)\tilde{a}(\xi)d\xi}^pdx\\
        \le &\brackets{\int_{2r<|x-x_0|<2\gamma_0\rho(x_0)}+\int_{|x-x_0|\ge2\gamma_0\rho(x_0)}} \sup_{0<t<\rho(x)}\varphi_t(x-x_0)^p\absvalue{\int_{Q_r(x_0)}\tilde{a}(\xi)d\xi}^pdx\\
        &+\int_{|x-x_0|>2r}\sup_{0<t<\rho(x)}\brackets{\int_{Q_r(x_0)}\left|(\varphi_t(x-x_0)-\varphi_t(x-\xi))\tilde{a}(\xi)\right|d\xi}^p dx\\
      \defeq &I_1+I_2+I_3.
    \end{align*}

The term $I_1$ can be controlled by
    \begin{align*}
        I_1\le \int_{|x-x_0|\le C\rho(x_0)}\frac{\rho(x_0)^{n-\frac{n}{p}}} {|x-x_0|^{np} }dx\le C
    \end{align*} for 
     $\frac{n}{n+1}<p< 1$, and when $p=1$,
    \begin{align*}
        I_1\le C\int_{2r<|x-x_0|\le C\rho(x_0)}\frac{1} {|x-x_0|^{n} }dx\left(\ln\left(1+\frac{r}{\rho(x_0)}\right)\right)^{-1}\le C.
    \end{align*}
    For the term $I_2$ we have 
    \begin{align*}
        I_2\le &\int_{|x-x_0|>C\rho(x_0)} \sup_{0<t<\rho(x)}\frac{C\rho(x)^p\rho(x_0)^{np-n}}{\brackets{1+\frac{|x-x_0|}{t}}^{pN}|x-x_0|^{(n+1)p}}dx\\\le& \int_{|x-x_0|>C\rho(x_0)}\frac{C\rho(x_0)^{(n+1)p-n}}{|x-x_0|^{(n+1)p}}dx\le C,
    \end{align*}
    by choosing a sufficient large $N$ for 
     $\frac{n}{n+1}<p\leq 1$.
    
    Finally, we have for 
     $\frac{n}{n+1}<p\leq 1$, 
    \begin{align*}
        I_3 \le &\int_{|x-x_0|>2r}\sup_{t>0}\brackets{\int_{Q_r(x_0)}\left|\varphi_t(x-x_0)-\varphi_t(x-\xi)\right|^2d\xi}^\frac{p}{2}\|\tilde{a}\|^p_2 dx\\
        \le &\int_{|x-x_0|>2r}\frac{r^{p+np-n}}{|x-x_0|^{(n+1)p}}dx\le C.
    \end{align*}
\end{proof}

\begin{lemma}
    Let $\frac{n}{n+1}<p\leq 1$ and $m$ be an approximate molecule defined as above, then $m\in H^p_\mathcal{L}(\rn)$. with $\|m\|_{H^p_\mathcal{L}}\le C$ uniformly.
\end{lemma}
\begin{proof}
     We temporarily write $Y_1=Q_{r}(x_0) $  and $Y_k=\Theta_{2^k r}(x_0)$ when $k\ge 2$ for convenience. We decompose $m$ into three parts by
     \begin{align*}
         m(x)=&\sum_{k=1}^\infty\brackets{\brackets{m(x)-\Avg{Y_k} m}\chi_{Y_{k+1}}(x)+\brackets{\frac{\chi_{Y_{k+1}}(x)}{|Y_k|}-\frac{\chi_{Y_k}(x)}{|Y_k|}}\int_{Y_k^\complement}md\xi}+\frac{\chi_{Y_1}(x)}{|Y_1|}\int_{\rn} md\xi
         \\\defeq & \sum_{k=1}^\infty (m_{1,k}(x)+m_{2,k}(x))+m_3(x).
     \end{align*}
     We can prove that $m_{1,k}, m_{2,k}$ are $H^p_\mathcal{L}$ atoms with $\|m_{1,k}(x)\|_{H^p_\mathcal{L}}+\|m_{2,k}\|_{H^p_\mathcal{L}}\le C2^{-k\delta^\prime}$ following the lines of \cite{NonCancelledAtom}. It is obvious that $m_3$ satisfies all the assumptions of Lemma \ref{noncancellingatom}, thus the proof is done.
\end{proof}

The Sobolev space $\dot{W}^{1,p}_V(\rn)=\dot{W}^{1,p}(\rn)\cap L^p\left(\rn,V^\frac{p}{2}dx\right)$ associated with $\mathcal{L}$ has a corresponding endpoint space, the Hardy-Sobolev space $H^{1,p}_V$ defined as 
\begin{equation*}
    H^{1,p}_V(\rn)=\left\{f \mid \|f\|_{H^{1,p}_V}\defeq\|\nabla f\|_{H^p}+\lnorm V^\frac{1}{2}f\rnorm_p<\infty \right\},
\end{equation*}

Adapting the same arguments of \cite[Thm. 4.4]{RegSchrodingerTao}, we can give a characteristic of   $H^{1,p}_V(\rn)$ as follows:
For $\frac{n}{n+1}<p\le 1$,  any distribution $g$ in $H^{1,p}_V(\rn)$ , we can write as
$$g=\sum_{k\in\mathcal{Z}}\lambda_k a_k,$$
the sum converges in $H^{1,p}_V(\rn)$ sense, where each $a_k$, is an atom
 satisfying the following conditions
\begin{itemize}
    \item [(i)]$\supp a_k(x) \subset Q_{r_k}(x_k)$, with $r_k<C\rho(x_k)$, %with $r\le \gamma_0\rho(x_k)$ for some $\gamma_0>1$;
    \item [(ii)]$\|\nabla a_k\|_2+\left\|V^{\frac 12}a_k\right\|_2\le C r_k^{\frac{n}{2}-\frac{n}{p}} $.
\end{itemize}
Moreover, $\{\lambda_k\}\in l^p$ and
$$\sum_{k\in\mathcal{Z}}|\lambda_k|^p\le C\|g\|^p_{H^{1,p}_V(\rn)}$$
with the constant $C$ independent of $g$. Also note that for $1<p<\infty$ the space $H^{1,p}_V(\rn)=\dot{W}^{1,p}_V(\rn)$ is exactly the homogeneous Sobolev space.

Finally we discuss the endpoint space serving as an alternative for $L^\infty$. Consider the Campanato space $\Lambda^\beta _{\mathcal{L},p} $ associated with $\mathcal{L}$  defined for $1\leq p<\infty$ as
\begin{align*}
    \Lambda^\beta _{\mathcal{L},p}(\rn)\defeq &\left\{f \mid\|f\|^p_{\Lambda^\beta _{\mathcal{L},p}}\defeq \sup_{r<\rho(x)}\frac{1}{r^{\beta }}\fint_{ Q_r(x)}\left|f(\xi)-\Avg{Q_r(x)} f\right|^pd\xi\right.\\&\left.+\sup_{r\ge \rho(x)}\frac{1}{r^{2\beta }}\fint_{ Q_r(x)} |f(\xi)|^p d\xi< \infty\right\}.
\end{align*}
We write $\Lambda^\beta _{\mathcal{L}}(\rn)=\Lambda^\beta _{\mathcal{L}}(\rn)$. From proposition 3 in \cite{BMOLbongioanni}, we know that 
$$\|f\|_{\Lambda^\beta _{\mathcal{L}}(\rn)}\sim \|f\|_{\Lambda^\beta _{\mathcal{L},p}(\rn)},\quad 1\leq p<\infty.$$
Note that $\Lambda^\beta_\mathcal{L}(\rn) $ is a subspace of $ \Lambda^\beta(\rn)$,  and the space is exactly the $BMO_\mathcal{L}$ space when $ \beta=0$. We write $\Lambda_\mathcal{L}^\beta=\Lambda_{\mathcal{L},2}^\beta$.

%\section{The Space $H^p_L$}\label{section: endpoint spaces and consequences}

%\subsection{Approximate Molecule}

%In this subsection we investigate more on the functions contained in the space $H^p_L $. 

\section{Estimates for Layer Potentials}\label{section: estimates for layer potentials}

\subsection{$ L^p$ Maximal Bounds for Layer Potentials}
We first study the $L^p$ bounds for the maximal function when $p>1$. The operator $\partial_t \calS  $ is clearly a Calder\'on-Zygmund operator due to its kernel estimates, thus it has the maximal function estimate
\begin{equation}
    \lnorm \calN (\partial_t \calS f)\rnorm_p\le C \brackets{1+\sup_{t>0}\lnorm \partial_t\calS f(\cdot,t)\rnorm_p}\|f\|_p
\end{equation}
for $ 1<p<\infty$. Denote the ``tangential total derivative'' $\nabla _{x,V}u=\left(\nabla_{\tan}u,V^\frac{1}{2}u\right)$, it is not so obvious that $\tilN(\nabla_{x,V}\calS f)$  also satisfies a similar estimate. However we can still prove its $L^p$ boundedness.
\begin{theorem}\label{maximalgradientlayerpotential}For $1 < p<\infty$, we have
\begin{equation*}
    \lnorm \tilN(\nabla_{x,V}\calS f) \rnorm_p \le C \sup_{t\ge 0}\lnorm \nabla_{x,V} \calS f(\cdot,t) \rnorm_p+C\lnorm \calN(\partial _t \calS f)\rnorm_p.
\end{equation*}    
\end{theorem}
\begin{proof}
    Write $u=\calS f$, $X^\prime=(x^\prime,t^\prime)$. Pick a Whitney cube $ Q_r(x^\prime,t^\prime)$ where $|x^\prime-x|<\gamma t^\prime$ and $r=\theta t^\prime$. We need to discuss different cases with respect to the side length $r$.

    \begin{itemize}
        \item [Case 1.]$\fint_{Q_{2\sigma r}(x^\prime)} V^\frac{1}{2} d\xi>\frac{1}{2\sigma  r} $, where $\sigma>1$ is a constant such that $2\sigma \theta<1$. By Caccioppoli's inequality and \rref{degiorginash} we have
    \begin{align*}
        &\brackets{\fint_{Q_r(X^\prime)}|\nabla_{x,V}u|^2 d\xi d\tau}^\frac{1}{2} \\\le &\frac{C}{r}\brackets{\fint_{Q_{\frac{3}{2} r}(X^\prime)}|u|^2 d\xi d\tau}^\frac{1}{2}\\\le& \frac{C}{ r}\brackets{\fint_{Q_{2  r}(X^\prime)}|u-C_Q|d\xi d\tau}+\frac{C|C_Q|}{r}
        \\\le &\frac{C}{ r}\fint_{Q_{2 r}(X^\prime)}|u-u(\xi,0)|d\xi d\tau+\frac{C}{ r}\fint_{Q_{2 r}(x^\prime)}|u(\xi,0)-C_Q|d\xi+\frac{C|C_Q|}{ r}\\
        \defeq& I_1+I_2+I_3.
    \end{align*}

 Taking $C_Q=\Avg{Q_{2 r}(X^\prime)}u(x,0)$, we can deal with $I_2$ by Poincar\'e inequality to get

 \begin{align}\label{dealing with I2}
     I_2\le C\fint_{Q_{2r}(x^\prime)}|\nbltg u(\xi,0)|d\xi\le C\mathcal{M}(\nbltg u(\cdot,0))(x),
 \end{align}
 and by \rref{localfeffermanphong} we have
 \begin{align*}
     I_3\le C\fint_{Q_{2r}(x^\prime)}|\nabla_V u(\xi,0)|d\xi\le C\mathcal{M}(\nabla_{x,V} u(\cdot,0))(x).
 \end{align*}
 
    For $I_1$ we have 
    \begin{align}\label{dealing with I1}
        |u(\xi,\tau)-u(\xi,0)|\le& \int_0^\tau |\partial_t u(\xi,\sigma) |d\sigma \le Cr\calN(\partial_t u)(\xi),
    \end{align}
thus integrating on $Q_{2 r}(X^\prime)$ we get  $I_1\le \mathcal{M}(\partial_t u)(x).$

        \item[Case 2.]$\fint_{Q_{2 r}(x^\prime)} V^\frac{1}{2} d\xi\le \frac{1}{2\sigma  r} $, we use the Caccioppoli's inequality with regard to $u-C_Q$, satisfying $\mathcal{L}(u-C_Q)=-V^\frac{1}{2}V^\frac{1}{2}C_Q=\DIV_V\left(0,V^\frac{1}{2}C_Q\right)^\dag$, where $\DIV_V =\left(\nabla,-V^\frac{1}{2}\right)$ . We have
        \begin{align*}
            &\brackets{\fint_{Q_r(X^\prime)}|\nabla_{x,V}u|^2 d\xi d\tau}^\frac{1}{2} \\\le &\frac{C}{r}\brackets{\fint_{Q_{\frac{3}{2} r}(X^\prime)}|u-C_Q|^2 d\xi d\tau}^\frac{1}{2}+{C}\brackets{\fint_{Q_{\frac{3}{2} r}(X^\prime)}V(\xi)d\xi}^\frac{1}{2}|C_Q|.
        \end{align*}

        If we can establish \begin{align}
            \brackets{ \fint_{Q_{\frac{3}{2} r}(X^\prime)}|u-C_Q|^2 d\xi d\tau}^\frac{1}{2}\le C\fint_{Q_{2  r}(X^\prime)}|u-C_Q|d\xi d\tau+{C}r\brackets{\fint_{Q_{2 r}(x^\prime)}Vd\xi}^\frac{1}{2}|C_Q|,\label{keyinequality31}
        \end{align}

        Then we get
        \begin{align*}
            &\brackets{\fint_{Q_r(X^\prime)}|\nabla_{x,V}u|^2 d\xi d\tau}^\frac{1}{2} \\\le &C\fint_{Q_{2 r}(X^\prime)}|u-u(\xi,0)|d\xi d\tau\\&+\frac{C}{ r}\fint_{Q_{2 r}(x^\prime)}|u(\xi,0)-C_Q|d\xi+{C}\brackets{\fint_{Q_{2 r}(x^\prime)}Vd\xi}^\frac{1}{2}|C_Q|\\
   \defeq& J_1+J_2+J_3.
        \end{align*}

       Taking $C_Q=\Avg{Q_{2 r}(x^\prime)}u$, while $J_1$ and $J_2$ can be handled similarly to \rref{dealing with I1} and \rref{dealing with I2} in the previous case, the term $J_3$ can be reduced by the property $V^\frac{1}{2}\in \mathcal{B}_{2q}$ to
        \begin{align*}
            J_3\le C\fint_{Q_{2 r}(x^\prime)}|\nabla_{x,V}u(\xi)| d\xi\le C\mathcal{M}(\nabla_{x,V}u(\cdot,0))(x).
        \end{align*}

        Now we start proving \rref{keyinequality31}.  For a fixed center denoted by $Y=(y,s)\in \rnp_+$ we write $Q_\varrho(Y)=Q_\varrho$ for the sake of convenience. For any $r_1,r_2\in\left[\frac{3}{2} r,2 r\right] $ and $r_1<r_2$, letting $\tilde{r}=\frac{r_1+r_2}{2}$ then using Poincar\'e-Sobolev inequality on the function $(u-C_Q)\phi(x)$ where $\phi$ is a smooth bump on $Q_{\tilde{r}}$ and being $1$ on $Q_{r_1}$, we have
        \begin{align*}
            &\brackets{\fint_{Q_{r_1}}|u-C_Q|^\frac{2(n+1)}{n-1} d\xi d\tau}^\frac{n-1}{2(n+1)}\\\le &Cr\brackets{ \fint_{Q_{\tilde{r}}}|\nabla u|^2 d\xi d\tau}^\frac{1}{2}+\frac{Cr}{\tilde{r}-r_1}\brackets{ \fint_{Q_{\tilde{r}}}| u-C_Q|^2 d\xi d\tau}^\frac{1}{2}\\%&+{C|C_Q|r}\brackets{\fint_{Q_{\tilde{r}}}V(\xi)d\xi}^\frac{1}{2}\\
            \le &\frac{Cr}{r_2-r_1}\brackets{\fint_{Q_{r_2}}| u-C_Q|^2 d\xi d\tau}^\frac{1}{2}+{C|C_Q|r}\brackets{\fint_{Q_{r_2}(y)}V(\xi)d\xi}^\frac{1}{2}\\
            \le &\frac{C\epsilon r}{r_2-r_1}\brackets{\fint_{Q_{r_2}}|u-C_Q|^\frac{2(n+1)}{n-1} d\xi d\tau}^\frac{n-1}{2(n+1)}+\frac{C}{\epsilon^\frac{n+1}{2}} \fint_{Q_{r_2}}|u-C_Q| d\xi d\tau\\&+{C|C_Q|r}\brackets{\fint_{Q_{r_2}}V(\xi)d\xi}^\frac{1}{2}\\
            \le &\frac{C\epsilon r }{r_2-r_1}\brackets{ \fint_{Q_{r_2}}|u-C_Q|^\frac{2(n+1)}{n-1} d\xi d\tau}^\frac{n-1}{2(n+1)}+Cr\brackets{\fint_{Q_{2r}}V(\xi)d\xi}^\frac{1}{2}|C_Q|\\&+\frac{C}{\epsilon^\frac{n+1}{2} } \fint_{Q_{r_2}}|u-C_Q| d\xi d\tau.
        \end{align*}
Write $f(\varrho)=\brackets{\frac{1}{\varrho^{n+1}} \int_{Q_{\varrho}(Y)}|u-C_Q|^\frac{2(n+1)}{n-1} d\xi d\tau}^\frac{n-1}{2(n+1)}$, taking $\epsilon=\frac{r_1-r_2}{2Cr}$, we have 
\begin{align*}
    f(r_1)\le &\frac{1}{2}f(r_2)+Cr\brackets{\fint_{Q_{2r}(y)}V(\xi)d\xi}^\frac{1}{2}|C_Q|\\+&\frac{C}{(r_2-r_1)^ {\frac{n+1}{2}  }r^{\frac{n+1}{2}} } \int_{Q_{2 r}(Y)}|u-C_Q| d\xi d\tau.
\end{align*}
From this, by the well-known result in \cite{GilbargTrudinger} regarding function inequalities in this form, we can get
\begin{align*}
            &\brackets{\fint_{Q_{3r/2}}|u-C_Q|^\frac{2(n+1)}{n-1} d\xi d\tau}^\frac{n-1}{2(n+1)}
\\\le &C\fint_{Q_{2  r}(X^\prime)}|u-C_Q|d\xi d\tau+{C}{r}\brackets{\fint_{Q_{2 r}(x^\prime)}Vd\xi}^\frac{1}{2}|C_Q|. \end{align*}
Thus, \rref{keyinequality31} is proved.
        
    \end{itemize}

    To conclude for both cases above, we can get 
    $$\tilN(\nabla_{x,V}u)(x)\le C\mathcal{M}\left(\nabla_{x,V}u(\cdot,0)\right)(x)+C\mathcal{M}(\calN(\partial_t u))(x)$$ and the desired inequality follows.
\end{proof}

As a consequence of Theorem 5.8 in \cite{MorrisTurner} and Theorem \ref{maximalgradientlayerpotential}, we have
 \begin{equation}\label{maxi}
        \lnorm\tilN(\nabla_V \calS f)\rnorm _{2} \le  C\|f\|_{2}.
    \end{equation}

\begin{remark}
    In fact \rref{maxi} has already been proved in \cite[Thm. 1.3]{MorrisTurner}, by noticing the uuniform $L^2$ bound of $\nabla_V\calS f(\cdot,0)$. Also note that in the proof of Theorem \ref{maximalgradientlayerpotential} we have no mention of local boundedness of $u$ in Case 2. Using similar methods  we can also get rid of local boundedness condition in Case 1.
\end{remark}    

\subsection{Single Layer Potential when $1-\varepsilon^\prime <p\le 2$}

\begin{theorem}\label{maximalhardy}There exists constant $C$ and $0<\varepsilon^\prime <\frac{1}{n+1}$ such that  for $1-\varepsilon^\prime <p\le 1 $,
    \begin{equation*}
        \lnorm\tilN(\nabla_V \calS f)\rnorm _{p} \le  C\|f\|_{H^p_\mathcal{L}}.
    \end{equation*}
\end{theorem} 
\begin{proof}
    Let $f=a(x)$ be an $H^p_\mathcal{L}$-atom supported on $Q_r(x_0)$, it suffices to prove  
$ \lnorm\tilN(\nabla_V \calS a)\rnorm _{p}\le C$ uniformly in $r$ and $x_0$. By (\ref{maxi}), we obtain that $$\brackets{\int_{Q_{cr}(x_0)}\labs\tilN(\nabla_V \calS a)(\xi)\rabs^p d\xi}^\frac{1}{p}\le Cr^{\frac{n}{p}-\frac{n}{2}}\brackets{\int_{Q_{cr}(x_0)}\labs\tilN(\nabla_V \calS a)(\xi)\rabs^2 d\xi }^\frac{1}{2}\le Cr^{\frac{n}{p}-\frac{n}{2}}\|a\|_2\le C .$$ 
    For the integral of $\tilN (\nabla_V \calS a)$ on dyadic annuli $\Theta_{2^kr} (x_0)$ , we split it into two terms, one ``close to the boundary'' and one ``far from the boundary''. Let
    \begin{equation}\label{definitiontruncatedmaximal}
        \begin{array}{cc}\displaystyle
            \tilN_1(w)(x)=\sup_{|x^\prime-x|<\gamma t^\prime, t^\prime<|x-x_0|}\fint_{Q_{\theta t^\prime}(x^\prime,t^\prime)} |w|^2 d\xi d\tau, \\\displaystyle \tilN_2(w)(x)=\sup_{|x^\prime-x|<\gamma t^\prime,t^\prime\ge|x-x_0|}\fint_{Q_{\theta t^\prime}(x^\prime,t^\prime)} |w|^2 d\xi d\tau ,
        \end{array}
    \end{equation}
    then $\tilN(u)\le \tilN_1(u)+\tilN_2(u)$. We first consider the term $\tilN_2(\nabla_V \calS a)$ which is more easy to handle, and analyze the behavior of  $\calS a$ for two cases with respect to the radius $r$.
    \begin{itemize}
        \item [Case 1. ]$ r\le \rho(x_0), \int_{\rn} a d\xi=0.$ This implies \begin{align*}
        |\calS a(x,t)|=\absvalue{\int_{Q_r(x_0)} \brackets{\Gamma(x,t|\xi,0)-\Gamma(x,t|x_0,0)} a(\xi)d\xi}\le \frac{C r^{\alpha+n-\frac{n}{p}}}{(|x-x_0|+|t|)^{n-1+\alpha}},
    \end{align*}
    provided that $|x-x_0|>Cr$ with $C>2$. Hence for $t^\prime\ge |x-x_0|>Cr$ and $|x^\prime-x|<\gamma t^\prime $, by Caccioppoli's inequality we have
\begin{align*}
    \brackets{\fint_{Q_{\theta t^\prime}(x^\prime,t^\prime)}\absvalue{\nabla_V \calS a}^2d\xi d\tau}^\frac{1}{2}&\le \frac{C}{t^\prime}\brackets{\fint_{Q_{\theta t^\prime}(x^\prime,t^\prime)}\absvalue{\calS a}^2d\xi d\tau}^\frac{1}{2}\\
    &\le \frac{Cr^{\alpha+n-\frac{n}{p}}}{|x-x_0|^{n+\alpha}},
\end{align*}
     which implies
        \begin{align*}
            \tilN_2(\nabla_V \calS a)(x)\le \frac{Cr^{\alpha+n-\frac{n}{p}}}{|x-x_0|^{n+\alpha}}.
        \end{align*}
        \item [Case 2.]$ r> \rho(x_0).$ We have
        \begin{align*}
            |\calS a(x,t)|=&\absvalue{\int_{Q_r(x_0)}\Gamma(x,t|\xi,0)a(\xi)d\xi}\\\le&\frac{C_N r^{n-\frac{n}{p}}}{(1+(|x-x_0|+t)m_V(x_0))^N(|x-x_0|+|t|)^{n-1}} 
        \end{align*}
        Similar to the previous case, we get
        \begin{align*}
            \tilN_2(\nabla_V \calS a)(x)\le \frac{Cr^{n-\frac{n}{p}}}{|x-x_0|^{n}(1+|x-x_0|m_V(x_0))^N}.
        \end{align*}
    \end{itemize}

    For the term $\tilN_1 (\nabla_V \calS a)$, we consider the average of $\nabla_V \mathcal{S}a$ on Whitney cubes  in two cases with regard to its side length $ \rho=\theta t^\prime$, where $|x^\prime-x|<\gamma t^\prime\le \gamma |x-x_0|$.
    \begin{itemize}
        \item [Case 1.]$\fint_{Q_{c\rho}(x^\prime)} V^\frac{1}{2} d\xi>\frac{1}{c\rho} $, we use Caccioppoli's inequality regarding $\calS a$ and split the right hand side directly as
        \begin{align*}
        &\brackets{\fint_{Q_{\theta t^\prime}(x^\prime,t^\prime)}\absvalue{\nabla_V \calS a}^2d\xi d\tau}^\frac{1}{2}\\
        \le& \frac{C}{\rho}\brackets{\fint_{Q_{c\rho}(x^\prime,t^\prime)}|\calS a -C_Q|^2d\xi d\tau}^\frac{1}{2}+\frac{C|C_Q|}{\rho}\\\le &\frac{C}{\rho}\brackets{\fint_{Q_{c\rho}(x^\prime,t^\prime)}|\calS a(\xi,\tau) -\calS a(\xi,0)|^2d\xi d\tau}^\frac{1}{2}\\&+\frac{C}{\rho}\brackets{\fint_{Q_{c\rho}(x^\prime)}|\calS a(\xi,0) -C_Q|^2d\xi }^\frac{1}{2}+\frac{C|C_Q|}{\rho}\\
        \defeq& C\brackets{I_1+I_2+I_3}.
    \end{align*}

    For $I_1$, we have 
    \begin{align*}
        |\mathcal{S} a(\xi,\tau)-\calS a(\xi,0)|\le &\int_{0}^\tau\absvalue{\int_{Q_r(x_0)}\partial_\sigma\Gamma(\xi,\sigma|\eta,0)a(\eta ) d\eta }d\sigma
    \end{align*}
    which can be dealt with when $r<\rho(x_0)$ by
   \begin{align*}
        |\mathcal{S} a(\xi,\tau)-\calS a(\xi,0)|\le &\int_{0}^{c\rho}\int_{Q_r(x_0)}\absvalue{\partial_\sigma\brackets{\Gamma(\xi,\sigma|\eta,0)-\Gamma(\xi,\sigma|x_0,0)}}|a(\eta )| d\eta d\sigma\\\le & \frac{C\rho r^{\alpha+n-\frac{n}{p}}}{|\xi-x_0|^{n+\alpha}},
    \end{align*}

    when $r\ge \rho(x_0)$,
    \begin{align*}
        |\mathcal{S} a(\xi,\tau)-\calS a(\xi,0)|\le &\int_{0}^{c\rho}\int_{Q_r(x_0)}\absvalue{\partial_\sigma\Gamma(\xi,\sigma|\eta,0)}|a(\eta )| d\eta d\sigma\\\le & \frac{C\rho r^{n-\frac{n}{p}}}{(1+|\xi -x_0|m_V(x_0))^N|\xi-x_0|^{n}},
    \end{align*}

    thus we have
    \begin{align*}
        \int_{\Theta_{2^kr}(x_0)}I_1^2 dx\le \left\{\begin{array}{cc}
             C2^{-k(n+2\alpha)}r^{n-\frac{2n}{p}}, &r\le \rho(x_0),\\
            \frac{C2^{-kn}r^{n-\frac{2n}{p}}}{(1+2^k)^N}, &r>\rho(x_0).
        \end{array}
            \right.
    \end{align*}

    For $I_2$, taking $C_Q=\Avg{Q_{c\rho}(x^\prime)}\calS a(\cdot,0)$, by Poincar\'e-Sobolev inequality we have
    \begin{align*}
        I_2\le \brackets{\fint_{Q_{c\rho}\left(x^\prime\right)}\absvalue{\nbltg\calS a}^\frac{2n}{n+2}dx}^\frac{n+2}{2n}\le \left(\mathcal{M}\brackets{|\nbltg \calS a|^\frac{2n}{n+2}}(x)\right)^\frac{n+2}{2n},
    \end{align*}
    thus by \rref{gradientdifferenceholder} and maximal function inequalities we have \begin{align*}
        \int_{\Theta_{2^kr}(x_0)}I_2^2 dx\le& C \int_{\Theta_{2^k r}(x_0)}|\nbltg \calS a|^2 dx\le \frac{C}{2^{3k} r^3}\int_{\tilde{\Theta}_{2^kr}(x_0,0)}|\calS a|^2dxdt
        \\\le &\left\{\begin{array}{cc}
             C2^{-k(n+2\alpha)}r^{n-\frac{2n}{p}}, &r\le \rho(x_0),\\
            \frac{C2^{-kn}r^{n-\frac{2n}{p}}}{(1+2^k)^N}, &r>\rho(x_0).
        \end{array}
            \right.
    \end{align*}

    For $I_3$, using local Fefferman-Phong inequality \rref{localfeffermanphong} we have
    \begin{align*}
        I_3=\absvalue{\fint_{Q_{c\rho}(x^\prime)}\calS a(x,0)dx}\le C\fint_{Q_{c\rho}(x^\prime)}|\nabla_V \calS a| dx\le \mathcal{M}(\nabla_V \calS a),
    \end{align*}
    hence similarly to $I_2$ we have
    \begin{align*}
        \int_{\Theta_{2^k r}}I_3^2dx\le& C \int_{\Theta_{2^k r}(x_0)}|\nabla_V \calS a|^2 dx\le \frac{C}{2^{3k}r^3} \int_{\tilde{\Theta}_{2^kr}(x_0,0)}|\calS a|^2dxdt
        \\\le &\left\{\begin{array}{cc}
             C2^{-k(n+2\alpha)}r^{n-\frac{2n}{p}}, &r\le \rho(x_0),\\
            \frac{C2^{-kn}r^{n-\frac{2n}{p}}}{(1+2^k)^N}, &r>\rho(x_0).
        \end{array}
            \right.
    \end{align*}
    
        \item [Case 2.]$\fint_{Q_{c\rho}(x^\prime)} V^\frac{1}{2} d\xi \le \frac{1}{c\rho}$, we need to use Caccioppoli's inequality with regard to $\calS a-C_Q$, satisfying $\mathcal{L}(\calS a-C_Q)=-\DIV_V(0,V^\frac{1}{2}C_Q)^\ast$.
        \begin{align*}
        &\brackets{\fint_{Q_{\rho}(x^\prime,t^\prime)}\absvalue{\nabla_V \calS a}^2d\xi d\tau}^\frac{1}{2}\\\le &\frac{C}{\rho }\brackets{\fint_{Q_{c\rho}(x^\prime,t^\prime)}|\calS a -C_Q|^2d\xi d\tau}^\frac{1}{2}+{C}\brackets{\fint_{Q_{c\rho}(x^\prime,t^\prime)}V(\xi)C_Q^2 d\xi }^\frac{1}{2}\\
        \le & \frac{C}{\rho }\brackets{\fint_{Q_{c\rho}(x^\prime,t^\prime)}|\calS a -\calS a(\xi,0)|^2d\xi d\tau}^\frac{1}{2}+\frac{C}{\rho }\brackets{\fint_{Q_{c\rho}(x^\prime)}|\calS a(\xi,0) -C_Q|^2d\xi }^\frac{1}{2}\\&+C\brackets{\fint_{Q_{c\rho}(x^\prime)}V(\xi)C_Q^2 d\xi }^\frac{1}{2}\\
        \defeq &J_1+J_2+J_3.
    \end{align*}

    For $J_1$ and $J_2$ the result is the same as Case 1, where $C_Q=\Avg{Q_{c\rho}(x^\prime)} \calS a$ is chosen the same.  For $J_3$, by the property $V^\frac{1}{2}\in \mathcal{B}_{2q}$ we know 

    \begin{align}
        I_3\le \frac{C|C_Q|}{r}\fint_{Q_{c\rho}(x^\prime)}V(\xi)^\frac{1}{2}d\xi\le C\fint_{Q_{c\rho}(x^\prime)}|\nabla_V \calS a|dx\le C\mathcal{M}(\nabla_V\calS a).
    \end{align}
    \end{itemize}

\end{proof}

\begin{theorem}\label{H^p estimate on slices}With the assumptions above we have
    \begin{equation*}
        \sup _{t>0}\brackets{\lnorm\nbltg \calS f(\cdot,t) \rnorm_{H^p}+\lnorm \partial_\nu \calS f(\cdot,t)\rnorm_{H^p_\mathcal{L}}} \le C\|f\|_{H^p_\mathcal{L}}.
    \end{equation*}
\end{theorem}
\begin{proof}
    It can be observed from the proof  of Theorem \ref{maximalhardy} that $\nabla_V \calS a$ satisfies the first three conditions of an approximate molecule, if $a$ is an $H^p_\mathcal{L}$ atom and $\frac{n}{n+\alpha}<p\le 1$. It is easy to see $\int_{\rn}\nbltg \calS f dx=0$, therefore it suffices to prove 
    \begin{equation*}
        \labs\int_{\rn} \partial_\nu \calS adx\rabs\le C(rm_V(x_0))^\delta r^{n-\frac{n}{p}} .
    \end{equation*}

    Denote $ H_t=\rn \times (t,\infty)$, $D_r(x)=Q_r(x,0) \cap H_t, Z_r(x)=Q_{2r}(x,0)\setminus Q_r(x,0)\cap H_t$. By divergence formula we know 
    \begin{align*}
        \labs\int_{\rn} \partial_\nu \calS a(\xi,t)d\xi\rabs=&\absvalue{\int_{H_t} V(\xi)\calS a(\xi,\tau)d\xi d\tau}\\
        =&\absvalue{\int_{H_t} \int_{Q_{r}(x_0)} V(\xi)\Gamma (\xi,\tau|\eta,0)a(\eta)d\eta d\xi d\tau}\\
        \le& \absvalue{\int_{D_{4r}(x_0)}\int_{Q_{r}(x_0)}V(\xi)\Gamma (\xi,\tau|\eta,0)a(\eta)d\eta d\xi d\tau}\\&+\sum_{k=2}^\infty\absvalue{\int_{Z_{2^k r}(x_0)}\int_{Q_{r}(x_0)}V(\xi)\Gamma (\xi,\tau|\eta,0)a(\eta)d\eta d\xi d\tau}\\
        \defeq &I_1+\sum_{k=2}^\infty I_k.
    \end{align*}
     For $I_1$, we have 
     \begin{align*}
         I_1\le & C\int_{D_{4r}(x_0)}\int_{Q_r(x_0)}\frac{V(\xi)|a(\eta)|}{(|\xi-\eta|+\tau)^{n-1}}d\eta d\xi d\tau\\
         \le & C\int_{Q_{4r}(x_0)}\int_{Q_r(x_0)}\frac{V(\xi)|a(\eta)|}{|\xi-\eta|^{n-2}} d\eta d\xi\\
         \le & C\left(\frac r{\rho(x_0)}\right)^{\delta_0}  r^{n-\frac{n}{p}}=C\left(\frac r{\rho(x_0)}\right)^{\delta_0+n(1-\frac 1p)} \rho(x_0)^{n-\frac{n}{p}}.
     \end{align*}
     For $I_k (k\ge 2)$ we consider two cases.
     \begin{itemize}
         \item [Case 1.]$ r\le \rho(x_0), \int_{\rn}ad\xi=0$.
    \begin{align*}
        I_k\le& \int_{Z_{2^k r}(x_0)}\int_{Q_{r}(x_0)}V(\xi)|\Gamma (\xi,\tau|\eta,0)-\Gamma(\xi,\tau|x_0,0)||a(\eta)|d\eta d\xi d\tau\\
        \le &Cr^{n-\frac{n}{p}}\int_{\Theta_{2^k r}(x_0)}\frac{V(\xi)r^\alpha d\xi d\tau}{(1+(|\xi-x_0|+\tau)m_V(x_0))^N(|\xi-x_0|+\tau)^{n-1+\alpha}}\\
        \le &\frac{C(rm_V(x_0))^\alpha}{2^{k\alpha}}r^{n-\frac{n}{p}},
    \end{align*}
         Thus we have 
         $$\displaystyle\sum_{k=2}^\infty I_k\le C \left(\frac r{\rho(x_0)}\right)^{\alpha+n(1-\frac 1p)} \rho(x_0)^{n-\frac{n}{p}}.$$
         \item [Case 2.]$\rho(x_0)< r\le \gamma_0\rho(x_0)$. 
     \begin{align*}
         I_k\le  C\int_{Z_{2^k r(x_0)}}\int_{Q_r(x_0)}\frac{V(\xi)|a(\eta)|}{(1+2^k rm_V(x_0))^N 2^{kn-k}r^{n-1}}d\eta d\xi d\tau\le \frac{Cr^{n-\frac{n}{p}}}{(1+2^k rm_V(x_0))^{N^\prime}}.
     \end{align*}
       Hence \begin{align*}
           \sum_{k=2}^\infty I_k\le \sum_{k=2}^\infty \frac{Cr^{n-\frac{n}{p}}}{(1+2^k)^N}\le Cr^{n-\frac{n}{p}}\le C\rho(x_0)^{n-\frac{n}{p}}.
       \end{align*}  
     \end{itemize}
     
    To conclude both cases, take $\delta =\min\{\alpha,\delta_0 \}$, for $\frac{n}{n+\delta}<p\le 1$, from $r\le \gamma_0\rho(x_0)$ it can be seen that $\partial_\nu \mathcal{S}a$ is an $H^p_\mathcal{L}$ approximate molecule.
     
\end{proof}

\subsection{Single Layer Potential when $2\le p < 2+\varepsilon$}\label{layer potential when 2+epsilon}

\begin{theorem}\label{2plusepsilon}There exists $\varepsilon>0$ such that
\begin{equation*}
    \sup_{t>0}\|\nabla_V\calS f(\cdot ,t)\|_p\le C\|f\|_p
\end{equation*}   
for $2\le p<2+\varepsilon$.
\end{theorem}
\begin{proof}
    Let $f\in L^2(\rn)$. For a fixed $ t_0>0$,  let $r=\theta_0 t_0,\, Q_{r}(x_0)$ be a cube in $\rn$, and split $f$ into $f=f\chi_{Q_{4r}(x_0)}+f \chi_{Q_{4r}(x_0)^\complement}\defeq f_1+f_2$, and define $u_j(x,t)=\calS f_j(x,t)$ for $j=1,2$.

For $f_1$ it is obvious that 
 \begin{align*}
     \fint_{Q_r(x_0)}|\nabla _{x,V} u_1(x,t_0)|^2 dx\le \frac{C}{r^n}\sup_{t>0}\|\nabla_{x,V}\calS f_1\|_2^2\le C\fint_{Q_{4r}(x_0)}|f|^2dx.
 \end{align*}
 
 Hence it suffices to prove the following lemma.
    \begin{lemma}\label{reverseholderlemma}There exists a $\delta>0$ such that
        \begin{align*}
            \brackets{
            \fint_{Q_r(x_0)}|\nabla _{x,V} u_2(x,t_0)|^2 dx}^\frac{1}{2}\le &C\brackets{\fint_{Q_{4r}(x_0)}|\nabla _{x,V} u(x,t_0)|^{2-\delta}dx}^\frac{1}{2-\delta}\\&+C\brackets{\fint_{Q_{4r}(x_0)}\brackets{|f|^2+\calN_\pm(\partial_t u)^2}dx}^\frac{1}{2},
        \end{align*}
        where $\calN_\pm$ refers to the double-sided nontangential maximal function.
    \end{lemma}

    \begin{proof}[Proof of Lemma \ref{reverseholderlemma}]
      As in the proof of Theorem \ref{maximalgradientlayerpotential} and \ref{maximalhardy}, we consider two cases with regard to the radius $r$.
      \begin{itemize}
          \item [Case 1.]$\fint_{Q_{2r}} V^\frac{1}{2}dx>  \frac{1}{2r} $. By Lemma \ref{basiclemma} we have
           \begin{align}
            \brackets{
            \fint_{Q_r(x_0)}|\nabla _{x,V} u_2(x,t_0)|^2 dx}^\frac{1}{2}\le&\frac{C}{r}\brackets{
            \fint_{t_0-cr} ^{t_0+cr} \fint_{Q_{2r}(x_0)}| u_2(x,t)|^2 dxdt}^\frac{1}{2}\nonumber\\
            \le &\frac{C}{r}\brackets{
            \fint_{t_0-cr} ^{t_0+cr} \fint_{Q_{2r}(x_0)}| u_2(x,t)-u_2(x,t_0)|^2 dxdt}^\frac{1}{2}\nonumber\\&+\frac{C}{r}\brackets{ \fint_{Q_{2r}(x_0)}| u_2(x,t_0)-C_Q|^2 dxdt}^\frac{1}{2}+\frac{C|C_Q|}{r}\label{derivationof2plusepsilon}\\\le &
            C(I_1+I_2+I_3).\nonumber
        \end{align}

        For $I_2$, setting $C_Q=\Avg{Q_{2r}}u_2(x,t_0)$, by Poincar\'e inequality we have 
        \begin{align*}
            I_2\le &\brackets{\fint_{Q_{2r}(x_0)}|\nabla _{\tan} u_2(x,t_0)|^{\frac{2n}{n+2}}dx}^\frac{n+2}{2n}\\
            \le & C\brackets{\fint_{Q_{2r}(x_0)}|\nabla _{\tan} u(x,t_0)|^{\frac{2n}{n+2}}dx}^\frac{n+2}{2n}+C\brackets{\fint_{Q_{2r}(x_0)}|\nabla _{\tan} u_1(x,t_0)|^{\frac{2n}{n+2}}dx}^\frac{n+2}{2n}\\
            \le &C\brackets{\fint_{Q_{4r}(x_0)}|\nabla _{\tan} u(x,t_0)|^{\frac{2n}{n+2}}dx}^\frac{n+2}{2n}+ C\fint_{Q_{4r}(x_0)}|f|^2dx.
        \end{align*}

  And by \rref{localfeffermanphong} we have 
\begin{align*}
    I_3\le& C\fint_{Q_{2r}(x_0)}|\nabla_{x,V}u_2|dx\le C\brackets{\fint_{Q_{2r}(x_0)}|\nabla _{x,V} u_2(x,t_0)|^{\frac{2n}{n+2}}dx}^\frac{n+2}{2n}\\
            \le & C\brackets{\fint_{Q_{2r}(x_0)}|\nabla _{x,V} u(x,t_0)|^{\frac{2n}{n+2}}dx}^\frac{n+2}{2n}+C\brackets{\fint_{Q_{2r}(x_0)}|\nabla _{x,V} u_1(x,t_0)|^{\frac{2n}{n+2}}dx}^\frac{n+2}{2n}\\
            \le &C\brackets{\fint_{Q_{4r}(x_0)}|\nabla _{x,V} u(x,t_0)|^{\frac{2n}{n+2}}dx}^\frac{n+2}{2n}+ C\fint_{Q_{4r}(x_0)}|f|^2dx.
\end{align*}
We can get a similar estimate to $I_2$.
Finally by the layer potential expression we have 
\begin{align*}
    |u_2(x,t)-u_2(x,t_0)|=&|\calS f_2(x,t)-\calS f_2(x,t_0)|\\
    \le &\absvalue{\int_{t_0}^{t}\absvalue{\int_{Q_{4r}(x_0)^\complement}\partial_\tau\Gamma(x,\tau|\xi,0)f(\xi)d\xi} d\tau}\\
    \le &\absvalue{\int_{t_0}^t \brackets{\calN_{\pm} (\partial_t\calS f)(x)+\calN_{\pm}(\partial_t \calS f_1)(x)} d\tau}\\\le &r\brackets{\calN_{\pm} (\partial_t\calS f)(x)+\calN_{\pm}(\partial_t \calS f_1)(x)},
\end{align*}
therefore \begin{align*}
    I_2\le& C\brackets{\fint_{Q_{2r}(x_0)}\brackets{\calN_{\pm} (\partial_t\calS f)(x)+\calN_{\pm}(\partial_t \calS f_1)(x)}^2dx}^\frac{1}{2}\\
    \le &C\brackets{\fint_{Q_{4r}(x_0)}\calN_{\pm} (\partial_t\calS f)(x)^2dx}^\frac{1}{2}+\frac{C}{r^{\frac{n}{2}}}\|f_1\|_2\\
    \le & C\brackets{\fint_{Q_{4r}(x_0)}\calN_{\pm} (\partial_t\calS f)(x)^2dx}^\frac{1}{2}+C\brackets{\fint_{Q_{4r}(x_0)}|f|^2dx}^\frac{1}{2}.
\end{align*}
  
          \item [Case 2.]$\fint_{Q_{2r}} V^\frac{1}{2}dx\le  \frac{1}{2r} $. Similar to \rref{derivationof2plusepsilon} we get

     \begin{align*}
            &\brackets{
            \fint_{Q_r(x_0)}|\nabla _{x,V} u_2(x,t_0)|^2 dx}^\frac{1}{2}\\
            \le &\frac{C}{r}\brackets{
            \fint_{t_0-cr} ^{t_0+cr} \fint_{Q_{2r}(x_0)}| u_2(x,t)-u_2(x,t_0)|^2 dxdt}^\frac{1}{2}\\&+\frac{C}{r}\brackets{\fint_{Q_{2r}(x_0)}| u_2(x,t_0)-C_Q|^2 dxdt}^\frac{1}{2}\\&+\frac{C|C_Q|}{r}\brackets{\fint_{Q_{2r}(x_0)}Vdx}^\frac{1}{2}\\ \defeq&(II_1+II_2+II_3).
        \end{align*}
         Here, $II_1$ and $II_2$ are the same as $I_1$ and $I_2$ respectively, the only different term is $II_3$. However by \rref{localfeffermanphong} we can still get 
          $$ II_3\le  C\fint_{Q_{2r}(x_0)}|\nabla_{x,V}u_2|dx. $$

       \end{itemize}
       Thus by the same type of argument, the desired reverse H\"older type inequality is satisfied.
    \end{proof}
    Using methods similar to \cite[pp. 18]{MethodOfLayerPotential}, Theorem \ref{2plusepsilon} is proved.
\end{proof}

\subsection{Double Layer Potential for $\Lambda^\beta_\mathcal{L}$ Data } \label{double layer potential lambda}
For the double layer potential operator, its behavior for $\Lambda^\beta_\mathcal{L}$ data $f$ is of great interest. In this section , without mentioning the Carleson estimates of its derivative, we only focus on its global H\"older continuity properties in the sense of being in Campanato-type space  $\Lambda^\beta_\mathcal{L}\left(\rnp_+\right)$.
\begin{theorem}\label{double layer potential Lambda beta}For $0\le \beta<{\alpha} $ we have
    \begin{equation*}
        \lnorm \mathcal{D}f\rnorm_{\Lambda^{\beta}_\mathcal{L}\left(\rnp_+\right)}\le C\|f\|_{\Lambda^\beta_\mathcal{L} (\rn)}.
    \end{equation*}
\end{theorem}
\begin{proof}
    Without loss of generality we assume $\|f\|_{\Lambda^\beta_\mathcal{L}}=1$. Choose a cube $$Q_r(x_0,t_0)=\left( Q_r(x_0)\times[t_0-r,t_0+r] \right)\cap \rnp_+.$$ We consider two cases as usual.
    
    \begin{itemize}
        \item [Case 1.]$r>\rho(x_0)$. We need to split $f$ into $f=f\chi_{Q_{4r}(x_0)}+f\chi_{Q_{4r}(x_0)^\complement} \defeq f_1+f_2$. For the compactly supported $f_1$, from \rref{L^2 boundedness of slices} we have
    \begin{align}
       \int_{Q_r(x_0,t_0)} |\mathcal{D}f_1| d\xi d\tau\le &Cr^{(n+1)/2}\left(\int_{Q_r(x_0,t_0)} |\mathcal{D}f_1|^2 d\xi d\tau\right)^{\frac 12}\nonumber\\
       \le &C r^{(n+2)/2}\sup_{t>0} \|\mathcal{D}\cdot (\cdot,t)\|_{2\rightarrow 2}\left(\int_{Q_{4r}(x_0)}|f|^2 dx\right)^{\frac 12}\le Cr^{n+1+\beta}.\label{basicLambdabeta}
    \end{align}
    For $f_2$ we have
    \begin{align*}
        &\int_{Q_r(x_0,t_0)} |\mathcal{D}f_2| d\xi d\tau \\\le &C \int_{Q_r(x_0,t_0)}\left|\int_{Q_{4r}(x_0
        )^\complement}\overline{\partial_{\nu^\ast,\eta}\Gamma^\ast(\eta,0|\xi,\tau)}f(\eta)d\eta\right| d\xi d\tau\\
        \le &C \int_{Q_r(x_0,t_0)}\sum_{k=2}^\infty\brackets{\int_{\Theta_{2^kr}(x_0)}|\nabla\Gamma^\ast(\eta,0|x,t)|^2 d\eta}^{\frac 12}\brackets{\int_{\Theta_{2^kr}(x_0)}|f|^2d\eta}^{\frac 12} d\xi d\tau\\
        \le &C_N\int_{Q_r(x_0,t_0)}\sum_{k=2}^\infty \frac{2^{-kn}r^{-n}}{(1+2^krm_V(x_0))^N} 2^{kn+k\beta}r^{n+\beta}d\xi d\tau\\\le& Cr^{n+1+\beta}.
    \end{align*}
        \item [Case 2.]$r\le \rho(x_0)$. Let $ \rho_0=\rho(x_0)$ and $2^{k_0}r\le \rho_0\le 2^{k_0+1}r$, we split $f$ in a different way following the lines in \cite{bmoldong}, as $f=f\chi _{Q_{2\rho_0}(x_0)}+f\chi _{\brackets{Q_{2\rho_0}(x_0)}^\complement}\defeq f_1+f_2$. For the term $f_2$, by \rref{gradientdifferenceholder} we have
        \begin{align*}
            &\int_{Q_r(x_0,t_0)} |\mathcal{D}f_2-\mathcal{D}f_2(x_0,t_0)|d\xi d\tau\\\le &C \int_{Q_r(x_0,t_0)}\sum_{k=k_0+1}^\infty\brackets{\int_{\Theta_{2^kr}(x_0)}\left|\nabla\brackets{\Gamma^\ast(\eta,0|\xi,\tau)-\Gamma^\ast(\eta,0|x_0,t_0)}\right|^2 d\eta}^{\frac 12}\\&\qquad\cdot\brackets{\int_{\Theta_{2^kr}(x_0)}|f|^2d\eta}^{\frac 12} d\xi d\tau\\
        \le &C\int_{Q_r(x_0,t_0)}\sum_{k=k_0+1}^\infty 2^{k(\beta-\alpha)}r^{\beta}d\xi d\tau\le Cr^{n+1+\beta}.
        \end{align*}

    Denote $\mathcal{D}_0$ the double layer potential operator with kernel $\overline{\partial_{\nu^\ast,\xi}\Gamma^\ast_0(\xi,0|x,t)}$, as defined in Section \ref{estimates fundamental solutions green}.  If $\beta=0$, using \rref{difference gradient fundamental solution} and noticing that $\int_{\Theta_{2^{-k}\rho_0}(x_0)}|f|^2 dx\le C(1+k)^2 2^{-kn}\rho_0^{n}$, note that  $\beta<\alpha$, we have
        \begin{align*}
          & \int_{Q_r(x_0,t_0)} \left|\left(\mathcal{D}-\mathcal{D}_0\right)f_1\right| d\xi d\tau \\\le & \int_{Q_r(x_0,t_0)} \left|\int_{Q_{2\rho_0}(x_0)}\overline{\partial_{\nu^\ast,\eta}\left(\Gamma^\ast(\eta,0|\xi,\tau)-\Gamma^\ast_0(\eta,0|\xi,\tau)\right)}f_1(\eta) d\eta\right| d\xi d\tau \\
          \le &\int_{Q_r(x_0,t_0)} \sum_{k=2}^{\infty}\brackets{\int_{\Theta_{2^{-k}\rho_0(x_0)}}\absvalue{\nabla(\Gamma^\ast(\eta,0|\xi,\tau)-\Gamma^\ast_0(\eta,0|\xi,\tau))}^2 d\eta}^{\frac 12}\\
          &\qquad\cdot\brackets{\int_{\Theta_{2^{-k}\rho_0(x_0)}}|f|^2 d\eta}^{\frac 12} d\xi d\tau\\
           \le &C\int_{Q_r(x_0,t_0)} \sum_{k=2}^{\infty}2^{-k\delta_0 +kn/2}\rho_0^{\delta_0 -n/2}\rho_0^{-\delta_0 }(1+k)2^{-kn/2}\rho_0^{n/2}  d\xi d\tau\\
           \le & Cr^{n+1}\sum_{k=2}^\infty(1+k) 2^{-k\delta_0} \le Cr^{n+1}.
        \end{align*}

        If $\beta>0$, following \cite{BMOLbongioanni}, we need to further decompose $f_1$ into $ f_1= f\chi_{Q_{2\rho_0}\setminus Q_{2r}}+f\chi_{Q_{2r}}\defeq f_{11}+f_{12}$. For $f_{11}$, using \rref{double difference integral gradient fundamental solution} and note that $\int_{\Theta_{2^{k}r}(x_0)}|f|^2 dx\le C 2^{kn}r^{n}\rho_0^{2\beta}$ we have
         \begin{align*}
 \int_{Q_r(x_0,t_0)} &\left|\left(\mathcal{D}-\mathcal{D}_0\right)f_{11}-\left(\mathcal{D}-\mathcal{D}_0\right)f_{11}(x_0,t_0)\right| d\xi d\tau \\
          \le & \int_{Q_r(x_0,t_0)} \left|\int_{Q_{2\rho_0\setminus Q_{2r}}}\overline{\partial_{\nu^\ast,\eta}\left(\left(\Gamma^\ast(\eta,0|\xi,\tau)-\Gamma^\ast_0(\eta,0|\xi,\tau)\right)\right.}\right.\\
          &\left.\left.\overline{-\left(\Gamma^\ast(\eta,0|x_0,t_0)-\Gamma^\ast_0(\eta,0|x_0,t_0)\right)}%\right
          \right) f_{11}(\eta) d\eta\right|^2 d\xi d\tau \\
          \le &\int_{Q_r(x_0,t_0)} \sum_{k=1}^{k_0}\left(\int_{\Theta_{2^{k}r}}|\nabla\left((\Gamma^\ast(\eta,0|\xi,\tau)-\Gamma^\ast_0(\eta,0|\xi,\tau))
          \right.\right.
          \\&\left.\left.\left.-(\Gamma^\ast(\eta,0|x_0,t_0)-\Gamma^\ast_0(\eta,0|x_0,t_0))\right)\right|^2 d\eta\right)^{\frac 12}\brackets{\int_{\Theta_{2^{k}r}}|f|^2 d\eta}^{\frac 12} d\xi d\tau\\
           \le &C\int_{Q_r(x_0,t_0)} \sum_{k=1}^{k_0}2^{k\delta_0 -k\alpha-kn/2}r^{\delta_0 -n/2}\rho_0^{-\delta_0 }2^{kn/2}r^{n/2}\rho_0^{\beta}  d\xi d\tau\\
           \le & Cr^{n+1}\sum_{k=1}^{k_0}2^{k\delta_0 -k\alpha}r^{\beta} 2^{k_0\beta-k_0\delta_0}\\
           \le & Cr^{n+1+\beta}2^{k_0(\beta-\alpha)}\le Cr^{n+1+\beta}.
        \end{align*}
        
For the term $f_{12}$, by $m_V(\xi)\sim m_V(x_0)$, note that $0<\beta<\delta_0$, we have
\begin{align*}
         & \int_{Q_r(x_0,t_0)} \left|\left(\mathcal{D}-\mathcal{D}_0\right)f_{12}\right| d\xi d\tau \\
          \le & \int_{Q_r(x_0,t_0)} \left|\int_{Q_{2r}(x_0)}\overline{\partial_{\nu^\ast,\eta}\left(\Gamma^\ast(\eta,0|\xi,\tau)-\Gamma^\ast_0(\eta,0|\xi,\tau)\right)}f_{12}(\eta) d\eta\right| d\xi d\tau \\
                        \le &\int_{Q_r(x_0,t_0)} \sum_{k=1}^{\infty}\brackets{\int_{\Theta_{2^{-k}r}}\absvalue{\nabla(\Gamma^\ast(\eta,0|\xi,\tau)-\Gamma^\ast_0(\eta,0|\xi,\tau))}^2 d\eta}^{\frac 12}\brackets{\int_{\Theta_{2^{-k}r}}|f|^2 d\eta}^{\frac 12} d\xi d\tau\\
           \le &C\int_{Q_r(x_0,t_0)} \sum_{k=1}^{\infty}2^{-k\delta_0 +kn/2}r^{\delta_0 -n/2}\rho_0^{-\delta_0 }(2^{-k}r)^{\frac n2}\rho_0^{\beta}  d\xi d\tau\\
           \le & Cr^{n+1+\beta}\left(\frac r{\rho_0}\right)^{\delta_0-\beta}\le Cr^{n+1+\beta}.
        \end{align*}

        As proved in \cite[Corollary 3.4]{MethodOfLayerPotential}, $\mathcal{D}_01=C^\prime$ where $C^\prime$ is a constants. Since $\Lambda_V^\beta\subset \Lambda^\beta$ for $0\leq\beta <\alpha$, we can get 
        \begin{align*}
            \int_{Q_r(x_0,t_0)} \left|\mathcal{D}_0f-c_Q\right| d\xi d\tau \le Cr^{n+1+\beta}
        \end{align*}
        for some $c_Q$, by methods similar to \cite{MethodOfLayerPotential}. Meanwhile noticing 
        \begin{align*}
            \int_{Q_r(x_0,t_0)} |\mathcal{D}_0f_2-\mathcal{D}_0 f_2(x_0,t_0)|d\xi d\tau\le Cr^{n+1+\beta},
        \end{align*}
        the proof finishes.
        
        \begin{remark}
            Note that following the same way of reasoning we can also get $$\sup_{t>0}\|\mathcal{D}f(\cdot,t)\|_{\Lambda^\beta_\mathcal{L}(\rn)}\le C\|f\|_{\Lambda^\beta_\mathcal{L}}.$$
        \end{remark}

    \end{itemize}
\end{proof}

\section{Rellich Estimates on the Boundary}\label{section: rellich estimates on the boundary}
In this section we consider $H^p$ Rellich estimates for $1-\varepsilon^ \prime<p\le 1$ on the boundary, providing a key tool for proving the invertibility of layer potentials. It is known in \cite[\S 4.3]{MorrisTurner} that $L^2$ Rellich estimates hold on the boundary if $A$ is a small perturbation of a block form or Hermitian matrix. However, we need to to propose a stronger assumption on $A$, namely $A$ is of block form or real symmetric, to make use of the estimates of Green functions and Neumann functions in order to reach our conclusions by Neumann and Green function estimates and decay property of atom data solutions.
\begin{theorem} There exists an $\varepsilon^\prime\in (0,\frac{1}{n+1})$, such that for $1-\varepsilon^\prime<p\le 1$, assuming $u$ satisfies $\mathcal{L}u=0, \,\tilN (\nabla_V u)\in L^p$ and $\partial_\nu u\in H^p_\mathcal{L}(\rn)$, we must have $u(x,0)\in H^{1,p}_V(\rn)$ with the estimate
    \begin{equation*}
        \|\nbltg u\|_{H^p}+\lnorm V^\frac{1}{2}u\rnorm _p\le C\|\partial_\nu u\|_{H^p_\mathcal{L}}.
    \end{equation*}
\end{theorem}
\begin{proof}
    We assume an $H^p_\mathcal{L}$ atom data $\partial_\nu u=a(x)$. By $L^2$ Rellich estimate, it is easy to see that $\|\nabla_{x,V}u\cdot \chi_{Q_{cr}(x_0)}\|_2\le  \|\partial_\nu u\|_2\le Cr^{n-\frac{n}{p}}$.
    
    The solution $u$ can be represented using the Neumann function by
    \begin{align*}
        u(x,t)=\int_{Q_r(x_0)}N(x,t|\xi,0)a(\xi)d\xi.
    \end{align*}
    
    Since $N(x,t|\xi,\tau)$ and  
$\Gamma(x,t|\xi,\tau)$ share similar bounds, with the same argument in Theorem \ref{maximalhardy} , we reach the decay estimate for $u$, namely
\begin{equation*}
    |u(x,t)|\le \left\{\begin{aligned}
        &\frac{C r^{\alpha+n-\frac{n}{p}}}{\brackets{|x-x_0|+t}^{n-1+\alpha}},& rm_V(x_0)\le 1,\\
        &\frac{C_N r^{n-\frac{n}{p}}}{(1+(|x-x_0|+t)m_V(x_0))^N\brackets{|x-x_0|+t}^{n-1}} ,&1<rm_V(x_0)\le \gamma_0
    \end{aligned}\right.
\end{equation*}
    for $|x-x_0|+t>cr$.

    Define the Lipschitz domain $\Omega _R=\{(x,t)||x-x_0|+t>R\}$, the cone surface $\Gamma_{R}=\partial\Omega_R\cap \rnp_+$ and the annuli $K_{R_1,R_2}=\{(x,t)|R_1<|x-x_0|+t<R_2\}$. The boundary value problem on $\Omega_R$ can be converted to boundary value problems with the same type of coefficients on the half space $\rnp_+$, by a variable substitution $(x,t)\mapsto\left(x,t-(R-|x-x_0|)_+\right)$ with uniformly bounded Lipschitz norm.  Following the proof of \cite{RegSchrodingerTao}, consider the Rellich estimate on $\partial\Omega_R$ with $2^{k-2}r\le R\le 2^{k-1}R $, we know that 
    \begin{align*}
        \int_{\Theta_{2^k r}(x_0)}\left|\tilN (\nabla _Vu)\right|^2dx\le C \lnorm\partial_{\nu({\Omega_R})} u\rnorm_2^2\le C\int_{\Gamma_R}|\nabla u|^2 dS.
    \end{align*}
    Integrating with respect to $R$ from $2^{k-2}r $ to $ 2^{k-1}r$, by Caccioppoli's inequality near the boundary, since the Neumann boundary condition is satisfied, we know that 
\begin{align*}
    \int_{\Theta_{2^k r}}\left|\tilN (\nabla _Vu)\right|^2dx\le& \frac{C}{2^k r}\int_{K_{2^{k-2}r,2^{k-1}r}}|\nabla u(x,t)|^2 dxdt\\\le& \frac{C}{2^{3k} r^3}\int_{K_{2^{k-3}r,2^{k}r}}|u(x,t)|^2dxdt\\
    \le &C2^{-2k(\alpha+n-\frac np)}(2^kr)^{n-\frac{2n}{p}}.
\end{align*}

    From that we know, taking $p>\frac{n}{n+\alpha}$, we can verify that $\nbltg u$ is a molecule in $H^p(\rn)$ and the desired estimate holds.
\end{proof}

\begin{theorem} Suppose $V\in \mathcal{B}_q(\rn)$ with  $q\ge\frac {n+1}2$. There exists an $\varepsilon^\prime\in (0,\frac{1}{n+1})$, such that for $1-\varepsilon^\prime<p\le 1$, assuming $u$ satisfies $\mathcal{L}u=0,\, \tilN (\nabla_V u)\in L^p$ and $u(x,0)\in H^{1,p}_V(\rn)$, we must have $\partial_\nu u\in H^p_\mathcal{L}(\rn)$ with the estimate
    \begin{equation*}
        \|\partial_\nu u\|_{H^p_\mathcal{L}}\le C\|\nbltg u\|_{H^p}+C\lnorm V^\frac{1}{2}u\rnorm _p.
    \end{equation*}
\end{theorem}
\begin{proof}
    Set $a(x)$ an atom in $H^{1,p}_V$ supported on $Q_r(x_0)$, we can write
    \begin{align*}
        u(x,t)=\int_{Q_{r}(x_0)}\partial_{\nu^\ast,\xi}G(x,t|\xi,0)a(\xi)d\xi.
    \end{align*}

    The first step is to prove an appropriate decay estimate for $u$. 

    If $A$ is of block form, we can perform a reflection argument to extend $G(x,t|\xi,\tau)$ to the cube $\tilde{Q}_{4r}(x_0,0) $ in $\rnp$ with the same estimate as \rref{estiamteGreenNeumann}. For $|x-x_0|+t>Cr$ where $C$ large enough, by Lemma \ref{basiclemma} we have
    \begin{align}\label{decayestimateGreen}
        |u(x,t)|\le &\brackets{\int_{Q_r(x_0)}\absvalue{ \nabla_\xi G(x,t|\xi,0)}^2 d\xi}^\frac{1}{2}\|a\|_2\nonumber\\
        \le& \frac{C}{r^{\frac 23}}\brackets{\int_{-2r}^{2r}\int_{Q_{2r}(x_0)}|G(x,t|\xi,\tau)-G(x,t|x_0,0)|^2 d\xi d\tau}^\frac{1}{2}r\|\nbltg a\|_2 \nonumber\\\le& \frac{Cr^{n-\frac{n}{p}+\alpha}}{\left(1+(|x-x_0|+t)m_V(x_0)\right)^N(|x-x_0|+t)^{n-1+\alpha}}.
    \end{align}

    If $A$ is real symmetric, in order to get the desired estimate for $\int_{Q_r(x_0)}\partial_\nu G(x,t|\xi,0)^2 d\xi$, denoting $w(\xi,\tau)=G(x,t|
    \xi,\tau)$, it suffices to prove%by the nonnegativity of Green function and comparison theorems (see for example \cite{WienerCriteriaDirichlet}), we can conclude
    \begin{align*}
        \int_{Q_r(x_0)}\tilN_r (\nabla w)^2(\xi) d\xi\le \frac{C}{r^3}\int_{Q_r(x_0,0)} G(x,t|\xi,0) ^2 d\xi d\tau,
    \end{align*}
    where $\tilN_r$ is the lower part of maximal function truncated at height $r$, namely $$\tilN_r (w)(x)=\sup_{|y-x|<\gamma s,s<r}\left( \fint_{Q_{\theta s}(y,s)}|w|^2 \right)^\frac{1}{2}.$$

    Let $\varphi(x,t)\in C_c^\infty(\rnp)$ be a smooth bump supported on $Q_{6r}(x_0,0)$, with $\varphi=1$ on $Q_{3r}(x_0,0)$ and $|\nabla \varphi|\le \frac{C}{r}$. Then we have
    \begin{align*}
        \varphi(y,s) w(y,s)=&\int_{Q_{6r}(x_0,0)}A(\xi)\left(G(y,s|\xi,\tau)\nabla w(\xi,\tau)\nabla\varphi(\xi,\tau)\right.\\&\left.+\nabla_{(\xi,\tau)}G(y,s|\xi,\tau) w(\xi,\tau) \nabla \varphi(\xi,\tau)\right)d\xi d\tau 
        \\\defeq & F_1(y,s)+F_2(y,s).
    \end{align*} 
    We talk about $F_2$ first, and $F_1$ can be handled the same way. Let $Q_l(y,s)$ be a cube where $l=\theta s$ and $|y-y_0|<s$ for some $y_0\in Q_r(x_0)$, then
    \begin{align*}
        &\fint_{Q_{l}(y,s)} |\nabla F_2(\eta,\sigma)|^2 d\eta d\sigma \\\le& \frac{C}{s^2}\fint_{Q_{l}(y,s)} | F_2(\eta,\sigma)|^2 d\eta d\sigma
        \\\le &\frac{C}{s^2 r^2}\sup_{(\eta,\sigma)\in Q_l(y,s)}\left( \int_{\Theta_{3r}(x_0,0)}\left| \nabla_{(\xi,\tau)}G(\eta,\sigma|\xi,\tau) w(\xi,\tau)  \right|d\xi d\tau \right)^2\\
        \le &\frac{C}{s^2 r^2}\sup_{(\eta,\sigma)\in Q_l(y,s)}\left( \int_{\Theta_{3r}(x_0,0)}\left| \nabla_{(\xi,\tau)}G(\eta,\sigma|\xi,\tau)  \right|^2 d\xi d\tau \right)\left( \int_{\Theta_{3r}(x_0,0)}w(\xi,\tau)  ^2 d\xi d\tau\right)\\
        \le &\frac{C}{s^2 r^4}\sup_{(\eta,\sigma)\in Q_l(y,s)}\left( \int_{\tilde{\Theta}_{3r}(x_0,0)}G(\eta,\sigma|\xi,\tau)^2 d\xi d\tau \right)\left( \int_{{\Theta}_{3r}(x_0,0)}  w(\xi,\tau)  ^2 d\xi d\tau\right)\\
        \le &\frac{Cr^{n-3}}{s^2 }\sup_{(\eta,\sigma)\in Q_l(y,s)}G_0(\eta,\sigma|A_r(x_0))^2 \int_{\Theta_{3r}(x_0,0)} w(\xi,\tau)  ^2 d\xi d\tau
    \end{align*}
    where $G_0$ is the Green function of the equation $-\DIV(A\nabla u)=0$, and $A_r(x_0)$ is a point in $\tilde{\Theta}_{3r}(x_0,0)$ with distance to both the boundary and $(x_0,0)$ are equivalent to constant multiple of $r$. The last step is supported by comparison theorem(see for example \cite{WienerCriteriaDirichlet}) and Harnack inequality. By the properties of harmonic measure $\omega$ in \cite[Corollary 1.3.6]{HarmonicAnalysisTechniqueKenig} we have%solvability of elliptic equation $-\DIV(A\nabla u)=0$ on bounded domains we have
%need improving
%\begin{align*}
   % \int_{Q_r(x_0)}\partial_\nu G_0(x,t|\xi,0) d\xi\le& %C \int_{\partial Q_{(2+\sigma)r}(x_0,0)}|\nbltg G_0(x,t|Q)|^2 dQ,
   % {Cr^{n-1}}G(x,t|x_0,cr)\\\le& \frac{Cr^{\alpha+n-1}}{\left(1+(|x-x_0|+t)m_V(x_0)\right)^N\left( |x-x_0|+|t| \right)^{n-1+\alpha}}
%\end{align*}
%integrating in $\sigma$ from $0$ to $1$ we get
%\begin{align*}
   % \int_{Q_r(x_0)}|\nabla G_0(x,t|\xi,0)|^2 d\xi\le& \frac{C}{r}\int_{Q_{3r}(x_0,0)}|\nabla G_0(x,t|\xi,\tau)|^2 d\xi d\tau\\\le& \frac{C}{r^3}\int_{Q_{4r}(x_0,0)}|G_0(x,t|\xi,\tau)-G_0(x,t|x_0,0)|^2 d\xi d\tau,
%\end{align*}
\begin{align*}
    \frac{1}{s}G_0(\eta,\sigma |A_r(x_0))\le \frac{C}{l^{n}}\omega^{A_{r}(x_0)}(Q_{l}(y))\le C\mathcal{M}(\partial_\nu G_0(y_0,0|A_r(x_0))),
\end{align*}
thus we have 
\begin{align*}
    &\int_{Q_r(x_0)}\tilN_r (\nabla w)^2(\xi) d\xi\\\le& Cr^{n-3} \int_{Q_r(x_0)}\partial_\nu G_0(\xi,0|A_r(x_0))^2d\xi  \int_{\Theta_{3r}(x_0,0)} w(\xi,\tau)  ^2 d\xi d\tau\\
    \le &\frac{C}{r^3} \int_{\Theta_{3r}(x_0,0)} w(\xi,\tau)  ^2 d\xi d\tau
\end{align*}
by the $L^2$ solvability of Dirichlet problem for $-\DIV(A(rx)\nabla u)=0$ on the half space proved in \cite{WeightedMaximalRegularityAuscher}, where the implicit constants are unifiorm in $r$. Therefore \rref{decayestimateGreen} also holds.

    With the decay behavior proved above we can see the first three criteria of molecule are justified for $\tilN(\nabla_V u)$.
    
    The second step is to verify that $\partial_\nu u $ satisfies the condition (iv) in the definition of approximate molecules in $H^p_\mathcal{L}(\rn)$. By Green's formula we have %We will prove it by appealing to the definition of $H^p_\mathcal{L}$ directly. Let $\Phi(X)$ be a bump function in $C_c^\infty(\rnp)$, supported on $Q_{2}(0)$ with $\varphi(X)=1$ when $x\in Q_1(0)$, and $\varphi(x)\ge 0$, $\int_{\rn}\varphi(x)dx=1$, and let $\varphi_t$ be its dilation. Suppose $\mathcal{P}_V f$ is defined as in \rref{definitionofPV}.

    %For the part of integral close to $x_0$ we have

    %\begin{align*}
     %   \brackets{\int_{Q_{2r}(x_0)}(\mathcal{P}_V \partial_\nu u)^p d\xi}^\frac{1}{p} \le r^{\frac{n}{p}-\frac{n}{2}}\brackets{\int_{Q_{2r}(x_0)}(\mathcal{P}_V \partial_\nu u)^2 d\xi}^\frac{1}{2}\le r^{\frac{n}{p}-\frac{n}{2}}\|\partial_\nu u\|_2\le C.
    %\end{align*}

\begin{align*}
    %&\int_{Q_{2r}(x_0)^\complement}\sup_{0<t<\rho(x)}\absvalue{\int_{\rn} \partial_\nu u(\xi) \varphi_t(x-\xi) d\xi}^pdx\\\le& \int_{Q_{2r}(x_0)^\complement}\sup_{0<t<\rho(x)}\absvalue{\int_{\rnp}A(\xi)\nabla u(\xi,\tau) \nabla\varphi_{t+\tau}(x-\xi)d\xi d\tau}^pdx\\&+\int_{Q_{2r}(x_0)^\complement}\sup_{0<t<\rho(x)}\absvalue{\int_{\rnp_+} V(\xi)u(\xi,\tau)\varphi_{t+\tau} (x-\xi)d\xi d\tau}^pdx\\
    %\le &\int_{Q_{2^Kr}(x_0)}V(x)|u(x,t)|dxdt+\int_{Q_{2^K r}(x_0)^\complement}V(x)|u(x,t)|dxdt\\
    \absvalue{\int_{\rn}\partial_\nu u dx}=&\absvalue{\int_{\rnp}V(x)u(x,t)dxdt}\\
    \le& \int_{Q_{2r(x_0,0)}}V(x)|u(x,t)|dxdt+\int_{Q_{2r}(x_0,0)^\complement}V(x)|u(x,t)|dxdt\\
   \defeq &I_1+ I_2.
\end{align*}
    By \rref{decayestimateGreen} it is easy to know that 
    \begin{align*}
        I_2\le &C r^{n-\frac{n}{p}+\alpha}\int_{\rnp}\frac{V(x)}{(1+(|x-x_0|+t)m_V(x_0))^N(|x-x_0|+t)^{n-1+\alpha}}dxdt \\\le &C\brackets{\frac{r}{\rho(x_0)}}^{\alpha+n-\frac{n}{p}}\rho(x_0)^{n-\frac{n}{p}}.
    \end{align*}
   For $I_1$, using H\"older inequality and the maximal function inequality \cite[Lemma A.2]{MethodOfLayerPotential} we have

     \begin{align*}
         I_1\le & \brackets{\int_{Q_{2r}(x_0,0)}V(x)^\frac{q}{2(q-1)}|u(x,t)|^\frac{q}{q-1}dxdt}^\frac{2q-1}{2q}\brackets{\int_{Q_{2r}(x_0,0)}V(x)^qdxdt}^\frac{1}{2q}\\
         \le &\brackets{ \int_{Q_{Cr}(x_0)}\tilN \left(V^\frac{1}{2}u\right) ^{\frac{2qn}{(2q-1)(n+1)}} dx}^\frac{(2q-1)(n+1)}{2qn}\brackets{\fint_{Q_{2r}(x_0)}V(x)dx}r^{\frac{n+1}{2q}},
     \end{align*}
     where we can take $ \frac{2qn}{(2q-1)(n+1)}=p$, namely $q=\frac{p(n+1)}{2p(n+1)-2n}$, to get

     \begin{align*}
         I_1\le &\lnorm\tilN\left(V^{\frac{1}{2}}u\right)\rnorm_p\brackets{r^2\fint_{Q_{cr}(x_0)}Vdx}^\frac{1}{2}r^{n-\frac{n}{p}}\\
         \le & C\brackets{\frac{r}{\rho(x_0)}}^\frac{\delta_0 }{2} r^{n-\frac{n}{p}}
     \end{align*}
     choose $\delta =
     \min\left\{ {\alpha} ,\frac{\delta_0 }{2},1-\frac{n+1}{2q}\right\}=\min\left\{ {\alpha} ,1-\frac{n+1}{2q}\right\}$ , for $\frac{n}{n+\delta}<p\le 1$ we can see $\partial_\nu u$ fits the definition of $H^p_\mathcal{L}$.
    %Consider a bump function $\phi(x)\in C_c^\infty(\rn)$,
\end{proof}

\begin{remark}
    Note that we have reached a stronger result than Theorem \ref{boundaryvaluedefinition} below, that $\partial_\nu u$ is a molecule for $H^{1,p}_V$ atom boundary data. Moreover, as the presence of Lemma \ref{basiclemma}, when the coefficient is independent of $t$ the average can be dropped.
\end{remark}

\section{Solution of the Equation}\label{section: solution of equation}

\subsection{Existence and Uniqueness}\label{subsection: existence and uniqueness of solution}
We are now ready to prove the Main Theorem \ref{mainthm} by finding the correct data $f$ for the layer potential according to the boundary data. 
\begin{theorem}\label{invertibility of layer potentials}
    Given $A_0(x)$ uniformly elliptic, there exist $\epsilon_0>0$ and $0<\varepsilon^\prime<\frac{1}{n+1}$ such that if the coefficients of operator $\mathcal{L}$ satisfy $A(x)=A_0(x)+\epsilon_0\tilde{A}(x)$  and $\|\tilde{A}\|_\infty \le 1$, and $V\in \mathcal{B}_q(\rn)$ with $q\ge\frac{n+1}{2}$, then the operators $\frac{I}{2}\pm \tilde{\mathcal{K}}$ are invertible in $H^p_\mathcal{L}$, and $\calS:H^p_\mathcal{L} \rightarrow H^{1,p}_V $ is invertible for $1-\varepsilon^\prime<p\le 2$.
\end{theorem}
\begin{proof}
    The boundary layer potential operators are continuous with respect to perturbations of coefficients, namely
\begin{align*}
    \left\|\left(\tilde{\mathcal{K}}_1-\tilde{\mathcal{K}}_2\right)f\right\|_{H_{\mathcal{L}}^p(\rn)}\le C\|A_1-A_2\|_{\infty}\|f\|_{H_{\mathcal{L}}^p(\rn)},
\end{align*}
and
\begin{align*}
    \left\|\nabla_x(\mathcal{S}_1-\mathcal{S}_2)f\right\|_{H^p(\rn)}\le C\|A_1-A_2\|_{\infty}\|f\|_{H^p_{\mathcal{L}}(\rn)},
\end{align*}
    where $\calS_j$ and $K_j$ are the operators corresponding to $A_j$ and $1-\varepsilon ^\prime<p\le 2$, which can be shown using similar methods in \cite{RegularityComplexCoefficients}, as the operator $\nabla_V \mathcal{L}^{-1}\DIV_V$ is bounded in $L^2$. For $ \left\|V^\frac{1}{2}(\mathcal{S}_1-\mathcal{S}_2)f\right\|_{p}$ and $p\le 1$, we consider atom $f=a$, $\supp a\subset Q_r(x_0)$, we can argue as in the proof of Theorem \ref{maximalhardy} to get 
    \begin{align*}
        \left\|\nabla_x (1+m_V(x_0)|x-x_0|)^N \calS a(x,0)\right\|_{H^p}\le C,
    \end{align*}
    by methods in \cite{RegularityComplexCoefficients} we have
    \begin{align*}
        \left\|\nabla_x (1+m_V(x_0)|x-x_0|)^N (\calS_1-\calS_2) a(x,0)\right\|_{H^p}\le C\|A_1-A_2\|_{\infty},
    \end{align*}
    thus 
    \begin{align*}
        &\left\|V^\frac{1}{2}(\mathcal{S}_1-\mathcal{S}_2)a\right\|_{p}\\\le& \left\|V^{1/2}(1+m_V(x_0)|x-x_0|)^{-N}\right\|_{n}\left\|(1+m_V(x_0)|x-x_0|)^N (\calS_1-\calS_2)a(x,0)\right\|_{\frac{np}{n-p}}\\
        \le &\left\|\nabla_x (1+m_V(x_0)|x-x_0|)^N (\calS_1-\calS_2)a(x,0)\right\|_{H^p}\\
        \le &C\|A_1-A_2\|_{\infty}
    \end{align*}
    and it follows that
    \begin{align*}
        \left\|V^\frac{1}{2}(x)(\calS_1-\calS_2)f(x,0)\right\|_{p}\le C\|A_1-A_2\|_{\infty}\|f\|_{H^p_\mathcal{L}}.
    \end{align*}
    
    Suppose $A_0$ is real symmetric, in this case we utilize the continuity argument. For $\mu \in [0,1]$ we write $A_{\mu,0}=\mu A_0 +(1-\mu)I$. %and $A_\mu=(1-\mu)A_0+\mu \epsilon_0 \tilde{A}$, 
    By the Rellich estimates and jump relations, for any $f\in H^p_\mathcal{L}(\rn)$ we have
    \begin{align*}
        \|f\|_{H^p_\mathcal{L}}\le C \|\partial_\nu^- {\cal S}^- f\|_{H^p_\mathcal{L}}+C\|\partial_\nu{\cal S}^+ f\|_{H^p_\mathcal{L}}\le C\|\partial_\nu ^\pm {\cal S} ^\pm f\|_{H^p_\mathcal{L}}\le C\|\calS f\|_{H^{1,p}_V}
    \end{align*}
    holds uniformly for all $\mu\in [0,1]$ and all $\calS$ associated with coefficient $A_\mu$ and $A_{\mu,0}$.
    
    By results regarding the classical Schr\"odinger equation $\mathcal{L}_Iu=-\Delta u+V(x)u=0$ (see for example \cite{CampanatoSobolevSchrodinger}), boundary operators $\frac{I}{2}\pm \tilde{\mathcal{K}}_I:L^2(\rn)\rightarrow L^2(\rn)$ and  $\calS_I=(-\Delta_x+V)^{-\frac{1}{2}}:L^2(\rn)\rightarrow \dot{W}^{1,2}_V(\rn)$ are invertible, where $\tilde{\mathcal{K}}_I$ and $\calS _I$ are the layer potential operators associated with $\mathcal{L}_I$. Take $a(x)$ an $H^p_\mathcal{L}$ atom and assume $f\in L^2$ such that $\partial_\nu \calS f=a$. Then through Rellich estimate we know that $\|\partial_\nu^-{ \cal S}^-_I f\|_{H^p_\mathcal{L}}\le C\|\nabla_{x,V}\calS_I f\|_{H^{1,p}_V}\le C\|\partial_\nu \calS_I f\|_{H^p_\mathcal{L}}\le C$. That implies $f=\partial_\nu^-{\cal S}^-_I f+\partial_\nu {\calS}^+_I f\in H^p_\mathcal{L}(\rn)$, thus the range of $\frac{1}{2}\pm \tilde{\mathcal{K}}_I:H^p_\mathcal{L}(\rn)\rightarrow H^p_\mathcal{L}(\rn)$ is dense. Similarly we can prove that $\calS:H^p_\mathcal{L}(\rn)\rightarrow H^{1,p}_V(\rn)$ has dense range. That finishes the proof of the invertibility of layer potential operators. 

    Suppose otherwise that $A$ is of block form. Since the Rellich estimates are given, it only remains to prove the density of ranges. As $A$ is of block form, $\Gamma(x,t|\xi,\tau)$ must be an even function with respect to $t-\tau$, causing $\tilde{\mathcal{K}}=0$. Then we can deduce that $\mathcal{S}f(\cdot,0)=(-\DIV(A\nabla)+V)^{-\frac{1}{2}}f$, then by \cite{BaileyPotential} it is onto $\dot{W}^{1,2}_V$. Thus arguing as the real symmetric case we have $\tilde{\mathcal{K}}:H^p_\mathcal{L}\rightarrow H^p_\mathcal{L}$ and $\mathcal{S}:H^p\rightarrow H^{1,p}_V$ have dense ranges, thus invertible. 
\end{proof}
%need improving

\begin{theorem}\label{uniqueness of elliptic solutions}
    Let $A_0$ is of block form or real symmetric and $A$ satisfies the assumptions above for $\epsilon_0>0$ sufficiently small dependent on elliptic constants, DG-N-M constants and $[\![V]\!]_q$. Then we have
      \begin{itemize}
      \item[(i)] If $u$ satisfies $(N)_p$ with $g=0$ in the sense of $H^p_\mathcal{L}$, then we must have $u(x,t)=0$.
        \item[(ii)]If $u$ satisfies $(R)_p$ with $F=0$ in the sense of $H^{1,p}_V$, then we must have $u(x,t)=0$.
        %\item[$(3^\circ)$]
    \end{itemize}
\end{theorem}
\begin{proof}

    We first prove (i). Let $ u_\epsilon(x,t)=u(x,t+\epsilon)$, then $\nabla_V u$ is locally $L^2$ in $\rnp_+$. Define the cutoff function $\eta_R\in C_c^\infty(\rnp)$ supported on $Q_{2R}(X_0)$ with $\eta_R|_{Q_R(X_0)}=1 $ and $|\nabla \eta_R|\le \frac{C}{R}$. By the explicit expression of solutions we have
    \begin{align*}
        &u_\epsilon(X)\eta_R(X)\\=&\int_{\rnp_+} A^\ast (Y)\nabla N(X|Y)u_\epsilon(Y) \nabla \eta_{R}(Y)dY-\int_{\rnp_+} N(X|Y)A(Y)\nabla u_\epsilon(Y)\cdot\nabla\eta_{R}(Y)dY\\
        &+\int_{\rn} N(X|\sigma)\partial_\nu u_\epsilon(\sigma) \cdot \eta_R(\sigma)d\sigma+\int_{\rn}N(X|\sigma) u_\epsilon(\sigma) \cdot \partial_\nu\eta_R(\sigma)d\sigma\\
        \defeq&I_1+I_2+I_3+I_4.
    \end{align*}

    It is clear to see $I_3\rightarrow 0 $ as $\epsilon\rightarrow 0$ for fixed $R$.  For $I_2$, by Whitney decomposition it is clear that
    \begin{align*}
        |I_2|\le& \int_{Q_{2R}(X)}\frac{C}{R^{n}} |\nabla u_\epsilon(\xi,\tau)| d\xi d\tau\le \frac{CR^{n+1}}{R^n}\brackets{\fint_{Q_{2R}(X)}|\nabla u|^{\frac{p(n+1)}{n}}d\xi d\tau}^{\frac{n}{p(n+1)}}\\\le& \frac{C}{R^{\frac{n}{p}-1}}\brackets{\int_{\rn}\tilN(\nabla u)^p}^{\frac{1}{p}}\le \frac{C}{R^{\frac{n}{p}-1}}
    \end{align*}
    which tends to $0$ uniformly in $\epsilon$ as $R\rightarrow\infty$.

    For $I_1$ and $I_4$, we need the estimate for $\|\calN (u)\|_{\tilde{q}}$ for some index $\tilde{q}$. Given any point $X^\prime \in \rnp_+$ we have ${t^\prime }^n\brackets{\fint_{Q_{\theta t^\prime}(X^\prime)}|\nabla u|^2 dxdt}^\frac{p}{2}\le C\int_{Q_{t^\prime}(x^\prime)}\tilN(\nabla u)^p dx\le C$, thus we have 
    \begin{align*}
        \sup_{X_1,X_2\in Q_{\theta^\prime t^\prime}(X^\prime)}|u(X_1)-u(X_2)|\le& 2C\brackets{\fint_{Q_{\theta t^\prime}(X^\prime)}|u-\Avg{Q_{\theta t^\prime}(X^\prime)}u|^2 d\xi d\tau}^\frac{1}{2}\\\le& Ct^\prime\brackets{\fint_{Q_{\theta t^\prime}(X^\prime)}|\nabla u|^2d\xi d\tau}^\frac{1}{2}\le \frac{C}{{t^\prime}^{\frac{n}{p}-1}}.
    \end{align*}

    From this, by summing over dyadic cubes, we know that $\displaystyle\lim_{\underset{t^\prime\rightarrow \infty}{X^\prime\in \Gamma_x}}u(x^\prime,t^\prime)=l(x)$ exists uniformly in $x\in \rn$ when $p<n$, where $\Gamma_x $ represents the cone with vertex $(x,0)$. Moreover
\begin{align*}
    |l(x)|^2\le \lim_{t^\prime\rightarrow\infty}\sup_{Q_{\theta^\prime t^\prime}(x,t^\prime)}|u(\xi,\tau)|^2 \le \lim_{t^\prime\rightarrow\infty}\frac{C_N}{(1+t^\prime m_V(x))^N}\int_{Q_{\theta^\prime t^\prime}(x,t^\prime)}|u(\xi,\tau)|^2 d\xi d\tau=0.
\end{align*}
    thus we can write $u(x,t)=\int_{t}^\infty \partial_\tau u(x,\tau) d\tau$, and the proof that $\|\calN (u)\|_{\frac{np}{n-p}}\le C$ follows from \cite[pp. 667]{NeumannSchTao}, thus we have $I_3,I_4\rightarrow 0$ uniformly in $\epsilon$ as $R\rightarrow \infty$. 

    Then we follow \cite{SmallCarleson} to give a proof of (ii). For any $X_0=(x_0,t_0)\in \rnp_+$, let $\eta_{\epsilon,R}\in C^\infty _c(\reals)$ supported on $\left[\frac{\epsilon}{2},2R\right]$, satisfying $\eta_{\epsilon,R}(t)=1$ in $[\epsilon,R]$ and $|\nabla \eta_{\epsilon,R} (t)|\le \frac{C}{\epsilon}$ for $t\in \left[\frac{\epsilon}{2},\epsilon\right]$, $|\nabla \eta_{\epsilon,R} (t)|\le \frac{C}{R}$ for $t\in [R,2R]$. Also define $\phi_R(x)\in C_c^\infty (\rn)$ supported on $ Q_{2R}(x_0)$, with $\phi_R(x)=1$ for $x\in Q_{R}(x_0)$, $|\nabla\phi_R(x)|\le \frac{C}{R}$. Finally we define cutoff function $\Phi_{\epsilon,R}(x,t)=\eta_{\epsilon,R}(t)\phi_R(x)$. %Define three domains $\Omega_j=\Omega_j^{\epsilon,R}(j=1,2,3)$ where $\nabla\Phi\ne 0$, as

    %\begin{align*}
     %   \Omega_1=&Q_{2R}(x_0)\times\left(\frac{\epsilon}{2},\epsilon\right),\\
      %  \Omega_2=&Q_{2R}(x_0)\times(R,2R),\\
       % \Omega_3=&\Theta_{R}(x_0)\times \brackets{\frac{\epsilon}{2},2R}.
    %\end{align*}
    
    With  $\Phi_{\epsilon,R}$ defined as above, by the definition of Green function we arrive at the formula 
    \begin{align*}
        &u(X)\Phi_{\epsilon,R}(X)\\=&\int_{\rnp_+} A^\ast (Y)\nabla G(X|Y)u(Y) \nabla \Phi_{\epsilon,R}(Y)dY-\int_{\rnp_+} G(X|Y)A(Y)\nabla u(Y)\cdot\nabla\Phi_{\epsilon,R}(\xi)dY
        \\=&\sum_{k=1}^{3}\left(\int_{\Omega_k}A^\ast (Y)\nabla G(X|Y)u(Y) \nabla \Phi_{\epsilon,R}(Y)dY-\int_{\Omega_k} G(X|Y)A(Y)\nabla u(Y)\cdot\nabla\Phi_{\epsilon,R}(\xi)dY\right)
        %\\\defeq&\sum_{k=1}^3\left(I_{1,k}+I_{2,k}\right).
    \end{align*}

    For $u$ near the boundary, we notice that by a well-known regarding $\dot{W}^{1,p}$ or $H^{1,p}$ space we know $u(x,0)$ can be defined as a function with $\nbltg u(x,t)\rightarrow \nbltg u(x,0)=0$ as $t\rightarrow 0$ , thus $u(x,0)=c$. Then by the convergence of $V^\frac{1}{2}u(x,t)$ to $0$ in the sense of distributions, we know that $c=0$. Thus we can write $u(x,t)=\int_{0}^t  \partial_t u(x,\tau)d\tau$. 
    By the definition of $\Phi_{\epsilon,R}$ and the pointwise estimate of Green function derived by \rref{differenceGreenNeumann}
    \begin{equation*}
        |G(x,t|\xi,\tau)|\le\frac{C_N\min\left\{ c(|x-\xi|+|t-\tau|)^\alpha,{t^\alpha}\right\} \min\left\{ c(|x-\xi|+|t-\tau|)^\alpha,{\tau^\alpha}\right\}}{(1+(|x-\xi|+|t-\tau|)m_V(x))^N(|x-\xi|+|t-\tau|)^{n-1+2\alpha}},
    \end{equation*}
     use the same type of estimate presented in \cite{SmallCarleson} we know that
     %\begin{align*}
      %   I_{2,1}\le \frac{C}{\epsilon}\int_{\Omega_2}\frac{\tau^\alpha}{T^{n-1+\alpha}}|\nabla u(\xi,\tau)|d\xi d\tau
     %\end{align*}
     %
     %Finally we get that 
     $u(X)=0$ by letting $\epsilon\rightarrow 0$ and then $R\rightarrow \infty$.
\end{proof}

\begin{remark}
    Note that the uniqueness is true once if the estimates for Green and Neumann functions \rref{estiamteGreenNeumann}, \rref{differenceGreenNeumann} hold, which applies to a broader class of matrices $A$ such that the weak solutions satisfy DG-N-M type local H\"older continuity on neighborhoods of points on the boundary, regardless of $A$ dependent of $t$ or not.
\end{remark}

Finally Theorem \ref{mainthm} is perfectly proved. 

\subsection{Boundary Values of Solutions}

We now determine the value of $u$ and $\nabla_V u$ on the boundary $\rn\times0$ for solutions of \rref{mainequation} with finite maximal function norms. The following theorem shows that the spaces $H^p_\mathcal{L}$ and $H^{1,p}_V$ for boundary data are necessary.

\begin{theorem}\label{boundaryvaluedefinition}
   $V\in \mathcal{B}_q(\rn)$ with  $q\ge\frac {n+1}2$.  Suppose $u$ satisfies $\mathcal{L}u=0$ and $\tilN(\nabla_V u)\in L^p(\rn)$. \begin{itemize}
        \item[(i)]There exists an $\varepsilon>0$, such that if $1<p<2+\varepsilon$, then there exists a function $u_0(x)\in H^{1,p}_V(\rn)$ and a function $g\in L^p(\rn)$ such that
        \begin{align}\label{variationalderivative}
            \int_{\rnp_+}(A\nabla u\nabla\Phi+Vu\Phi)dxdt=\int_{\rn}g\varphi dx
        \end{align}
        for every $\varphi\in C_c^\infty(\rn)$ and $\Phi(x,t)\in C_c^\infty(\rnp)$ such that $\Phi(x,0)=\varphi(x)$. And the following convergences hold as $t\rightarrow 0$:
    \begin{align*}
        u(x,t)\rightarrow u_0(x),\, &\textup{n.t. a.e.}x\in \rn,\\
        \fint_t^{2t} \nbltg u(\cdot,\tau)  d\tau &\overset{L^p}{\rightharpoonup} \nbltg u_0,\\
         \partial_\nu u(\cdot,t)  d\tau &\overset{L^p}{\rightharpoonup} g,
    \end{align*}
    with the bound
    \begin{align*}
         \| u_0\|_{\dot{W}^{1,p}_V}+\|g\|_{p}\le \|\tilN(\nabla_V u)\|_{p}.
    \end{align*}
    \item [(ii)]There exists a $\delta>0$ such that if $\frac{n}{n+\delta}<p\le 1$, then there exists a function $u_0(x)\in H^{1,p}_V(\rn)$ and a function $g\in H^p_{\mathcal{L}}(\rn)$ such that \rref{variationalderivative} holds and 
    \begin{align*}
        u(x,t)\rightarrow u_0(x),\, &\textup{n.t. a.e.}x\in \rn,\\
        \fint_t^{2t} \nbltg u(\cdot,\tau) d\tau &\overset{\mathcal{S}^\prime}{\rightharpoonup} \nbltg u_0,\\
        \partial_\nu u(\cdot,t)  &\overset{\mathcal{S}^\prime}{\rightharpoonup} g
    \end{align*}
    as $t\rightarrow 0$, where the last two convergences are interpreted in the sense of tempered distribution, with the following bound
    \begin{align*}
        \| u_0\|_{H^{1,p}_V}+\|g\|_{H^p_\mathcal{L}}\le \|\tilN(\nabla_V u)\|_{p}.
    \end{align*}
    \end{itemize}
\end{theorem}
\begin{proof}
    First we need to prove that $u(x)$ is locally Lipschitz continuous in the cone region, namely
    \begin{align}\label{local boundedness of u}
        |u(x_1,t_1)-u(x_2,t_2)|\le C \tilN(\nabla_V u)(x)\brackets{|x_1-x_2|+|t_1-t_2|}
    \end{align}
    for $|x_j-x|<t_j,\,\,j=1,2$. Write $X_j=(x_j,t_j) $, $X^\prime=(x^\prime,t^\prime)$ and set a Whitney cube $Q_r=Q_{r}(X_1)$ where $r=\theta t_1$, and $Q_r=Q_r(x_1)$.  As usual we consider two different case for $r$. 

    \begin{itemize}
        \item [Case 1.] $4r^2\fint_{Q_{2r}} Vdx> 1 $. Taking $C_Q=\Avg{Q_{2r}}u$, by Poincar\'e inequality and \rref{localfeffermanphong} we have
\begin{align*}
    \sup_{Q_{r}}|u|\le& C\brackets{\fint_{Q_{2r}} |u-C_Q|^2}^\frac{1}{2}+C|C_Q|\\
    \le &Cr\brackets{\fint _{Q_{2r}}|\nabla u|^2}^\frac{1}{2}+Cr\brackets{\fint_{Q_{2r}}|\nabla_V u|^2}^\frac{1}{2}\\
    \le &Cr \tilN(\nabla_V u)(x).
\end{align*}
        
        \item [Case 2.]$ 4r^2\fint_{Q_{2r}} Vdx\le  1$. To demonstrate \rref{local boundedness of u} it suffices to show that $ |u(x^\prime,t^\prime)-C_Q|\le  Cr\tilN(\nabla_V u)(x) $
    for all $(x^\prime,t^\prime)\in Q_r$ and a suitable constant $C_Q$.

    Take a bump function $\psi\in C^\infty_c(\rnp) $, $\supp\psi\subset Q_{2r}$ and $\psi|_{Q_{\frac{3}{2}r}}=1$ with $|\nabla\psi|\le \frac{C}{r}$, we have
    \begin{align*}
        &\left|(u(X^\prime)-C_Q)\psi(X^\prime)\rabs\\=&\labs\int_{\rnp_+}\left(A(\xi) \nabla_\xi\Gamma(X^\prime |\xi,\tau) \nabla\psi(\xi,\tau)\cdot(u(\xi,\tau)-C_Q)\right.\right.\\&\left.\left.-\Gamma(X^\prime|\xi,\tau)\brackets{A\nabla u \nabla\psi+C_QV(\xi)\psi}\right) d\xi d\tau\right|\\
        \le & \frac{C}{r}\int_{Q_{2r}}|\nabla_\xi\Gamma(X^\prime |\xi,\tau) |\left|u(\xi,\tau)-C_Q\right|d\xi d\tau+\frac{C}{r}\int_{Q_{2r}}|\Gamma(X^\prime |\xi,\tau)||\nabla u|d\xi d\tau\\&+C|C_Q|\int_{Q_{2r}}|\Gamma(X^\prime |\xi,\tau)| V(\xi) d\xi d\tau\\
        \defeq&I_1+I_2+I_3.
    \end{align*}

    Taking $C_Q=\Avg{Q_{2r}}u$, by Poincar\'e inequality we have
    \begin{align*}
        I_1\le &Cr^{n+1}\brackets{\fint_{Q_{2r}}|\nabla u|^2 }^\frac{1}{2}\brackets{\fint_{Q_{2r}\setminus Q_{\frac{3}{2}r}}|\nabla_\xi \Gamma(X^\prime|\xi,\tau)|^2 d\xi d\tau}^\frac{1}{2}\\
        \le &Cr^{n}\brackets{\fint_{Q_{2r}\setminus Q_{\frac{3}{2}r}}| \Gamma(X^\prime|\xi,\tau)|^2 d\xi d\tau}^\frac{1}{2}\brackets{\fint_{Q_{2r}}|\nabla u|^2 }^\frac{1}{2}\\
        \le &Cr\brackets{\fint_{Q_{2r}}|\nabla u|^2 }^\frac{1}{2}.
    \end{align*}

    For $I_2$ we have 
    \begin{align*}
        I_2 \le &Cr^{n}\brackets{\fint_{Q_{2r}}|\nabla u|^2 }^\frac{1}{2}\brackets{\fint_{Q_{2r}\setminus Q_{\frac{3}{2}r}}| \Gamma(X^\prime|\xi,\tau)|^2 d\xi d\tau}^\frac{1}{2}\\
        \le &Cr\brackets{\fint_{Q_{2r}}|\nabla u|^2 }^\frac{1}{2},
    \end{align*}

    And the term $I_3$ can be dealt with by using \rref{localfeffermanphong}, namely
    \begin{align*}
        I_3\le& C|C_Q| \int_{Q_{2r}}\frac{V(\xi)}{|\xi-x^\prime|^{n-2}}d\xi\\\le& Cr^{2-n}\int_{Q_{2r}}V  d\xi\brackets{\fint_{Q_{2r}}|u|^2}^\frac{1}{2}
        \\\le &Cr\brackets{r^{1-n}\int_{Q_{2r}}V  d\xi d\tau}^\frac{1}{2}\brackets{\fint_{Q_{2r}}Vd\xi \cdot\fint_{Q_{2r}}|u|^2}^\frac{1}{2}
        \\
        \le& Cr\brackets{\fint_{Q_{2r}}|\nabla_V u|^2}^\frac{1}{2}.
    \end{align*}
    \end{itemize}
    
    By \rref{local boundedness of u} we can meaningfully define $u(x,0)$. Its derivative $\nbltg u(x,0)\in L^p(\rn)$ exists as an element in $H^p(\rn)$ as a consequence of \cite{HardySobolevPointwise}, since $u(x,0)$ satisfies \begin{align*}
        |u(x_1,0)-u(x_2,0)|\le C\left(\tilN(\nabla_V u)(x_1)+\tilN(\nabla_V u)(x_2)\right)|x_1-x_2|.
    \end{align*}

    In addition, we notice that
    \begin{align}
        &\int_{\rn} \fint_{t}^{2t} \left(\nbltg ,V^\frac{1}{2}\right)(u(\xi,\tau)-u(\xi,0))h(\xi)d\tau d\xi \nonumber\\=&\int_{\rn} \fint_{t}^{2t} (u(\xi,\tau)-u(\xi,0))\left(-\DIV ,V^\frac{1}{2}\right)h(\xi)d\tau d\xi.\label{weakconvergence1}
    \end{align}

    For $1<p<2+\varepsilon$ it is easy to see that $\lnorm \fint_t^{2t} \nabla_V u d\tau\rnorm_p\le \lnorm \tilN (\nabla_V u)\rnorm_p<\infty$, by invoking the H\"older inequality as well as the reverse H\"older inequality for $|\nabla_V u|$. Thus each sequence $t_k$ has a subsequence $t_{k_l}$ such that $\fint_t^{2t} \nabla_V u d\tau $ converges weakly to some $g\in L^p(\rn)$. Since the right hand side of \rref{weakconvergence1} tends to $0$ when $t\rightarrow 0$ for any fixed $h\in C_c^\infty$, it must be true that $ \fint_t^{2t} \left(\nbltg,V^\frac{1}{2}\right) u d\tau\rightharpoonup \left(\nbltg,V^\frac{1}{2}\right) u(x,0)$ in $L^p$.

    For $\frac{n}{n+1}<p\le 1$, we can see $ \fint_t^{2t} \left(\nbltg,V^\frac{1}{2}\right) u d\tau\rightharpoonup \left(\nbltg,V^\frac{1}{2}\right) u(x,0)$ in the sense of tempered distributions by combining \rref{weakconvergence1} with  $ \|u(\cdot,t)\|_{\frac{np}{n-p}}\le C  $ as a conseqence of \rref{local boundedness of u}.% only need to prove that $ \nbltg u(x,0)\in H^p$.

    Next we consider the conormal derivative when $\frac{n}{n+1}\frac{2q}{2q-1}<p<2+\varepsilon$. For $t\ge 0$ we define a linear functional on $\mathcal{S}$, i.e. a tempered distribution, as 
    \begin{align*}
        F_t(\phi)=\int_t^\infty\int_{\rn} \left(A(\xi)\nabla u(\xi,\tau)\nabla\Phi(\xi,\tau)+V(\xi) u\Phi \right)d\xi d\tau,
    \end{align*}
    where $\Phi(x,t)\in \calS(\rnp)$ such that $\Phi(x,0)=\phi(x)$. By the weak solution property and \rref{integrability of V} we can know that $F_t$ is well-defined. For fixed $\phi$, as $t\rightarrow0$ we can see $F_t \rightarrow F_0$. Indeed, we observe that $\nabla_V u\in L^{p\frac{n+1}{n}}\left(\rnp_+\right)$ for $\frac{n}{n+\delta}<p\le 1$ and $\nabla_V u\in L^p(\rn\times(0,T))$ for any $T>0$ for $1<p<2+\varepsilon$, also $V^{\frac{1}{2}}(1+|x|m_V(0))^{-N}\in L^p$ for every $1\le p\le n+1$ for sufficiently large $N$ according to \rref{integrability of V}, then by taking a specific $\Phi(x,t)=\phi(x)\eta(t)$ with $\eta(t)\in C_c^\infty(\reals)$ and $\eta|_{0\le t\le 1}=1$, we can prove the convergence. Thus the weak convergence when $p>1$ and convergence in distributions when $\frac{n}{n+1}\frac{2q}{2q-1}<p
    \le 1$ can be obtained. Finally we prove that $F_0\in H^p_\mathcal{L}(\rn)$ as a distribution by referring directly to the definition given via $\mathcal{P}_\rho$ as in \rref{definitionofPV}. Consider%$|F_{2t}-F_t|\le Ct$ with $t$ sufficiently small. Thus 
    \begin{align*}
        &\labs \int_{\rn} \varphi_t(x-\xi)\partial_\nu u(\xi,0)dx\rabs\\= &\labs\int_{\rnp_+} \left(A(\xi)\nabla u(\xi,\tau) \nabla \Phi_t(x-\xi,\tau) +V(\xi)u(\xi,\tau)\Phi_t(x-\xi,\tau) \right)d\xi d\tau\rabs\\
        \le &\frac{C}{t^{n+1}}\int_{Q_t(x,0)}|\nabla u(\xi,\tau)|d\xi d\tau+\frac{C}{t^n}\int_{Q_{t}(x,0)} V(\xi)|u(\xi,\tau)|d\xi d\tau\\
    \defeq &J_1+J_2.
    \end{align*}

    For $J_1$, by maximal function inequality we have
    \begin{align}
        J_1\le& C\brackets{ \fint_{Q_t(x,0)}|\nabla u|^\frac{(p-\epsilon)(n+1)}{n}d\xi d\tau}^\frac{n}{(n+1)(p-\epsilon)}\le \frac{C}{t^{\frac{n}{p}}}\brackets{\int_{Q_{ct}(x,0)}\tilN (\nabla_V u)^{p-\epsilon}}^\frac{1}{p-\epsilon}\nonumber\\
        \le &C\brackets{\mathcal{M}\brackets{\tilN (\nabla_V u)^{p-\epsilon}}(x)}^\frac{1}{p-\epsilon}\label{I_1 maximal controlling}
    \end{align}
    for some $\epsilon>0$ such that $p-\epsilon>\frac{n}{n+1}$.

    For $J_2$ we have
    \begin{align}
        J_2\le & \frac{C}{t^n}\brackets{ \int_{Q_t(x,0)}\labs V^\frac{1}{2} u\rabs^{\frac{2q}{2q-1}}d\xi d\tau}^{1-\frac{1}{2q}}\brackets{\int_{Q_t(x,0)}V^q d\xi d\tau}^{\frac{1}{2q}}\nonumber\\
        \le& \frac{C}{t^{(n+1)\brackets{1-\frac{1}{2q}}}}\brackets{\int_{Q_{ct}(x,0)}\tilN (\nabla_V u)^{\frac{2q}{2q-1}\frac{n}{n+1}}}^{\brackets{1-\frac{1}{2q}}\brackets{1+\frac{1}{n}}}\brackets{t^2\fint_{Q_t(x)} Vd\xi }^\frac{1}{2}\nonumber\\
        \le &C\brackets{ \mathcal{M}\brackets{\tilN\left(\nabla_V u\right)^{\frac{2q}{2q-1}\frac{n}{n+1}}}(x)}^{\brackets{1-\frac{1}{2q}}\brackets{1+\frac{1}{n}}}\label{I_2 maximal controlling}
    \end{align}
    since $t<\rho(x)$.
    
    Therefore, choosing $\varepsilon=p-\frac{n}{n+1}\frac{2q}{2q-1}>0$ we have 
    \begin{align*}
        \mathcal{P}_\rho(\partial_\nu u(x,0))\le C\brackets{ \mathcal{M}\brackets{\tilN\left(\nabla_V u\right)^{p-\varepsilon}}(x)}^\frac{1}{p-\varepsilon}
    \end{align*}
    which gives $\lnorm\partial_\nu u(x,0)\rnorm_{ H^p_\mathcal{L}}\le C\lnorm \tilN(\nabla_V u) \rnorm_p$.
    
    Let $\delta=1-\frac{n+1}{2q}>0$, the conclusion (ii) follows.
\end{proof}
\begin{remark}
    Note that the conclusion holds also for coefficient dependent on $t$, as long as $A$ satisfies the uniform elliptic assumption. With slight modification on the proof of $F_t(\phi)\rightarrow 0$ as $t\rightarrow 0$, the conclusion also holds. Moreover, if the coefficient is $t$-independent, $\fint_{t}^{2t}\nbltg u$ can be replaced by $\nbltg u(\cdot,t)$ as the presence of Lemma \ref{basiclemma}.
\end{remark}

%\appendix
\subsection{Regularity of Complex Solutions}\label{Appendix A}

    \begin{theorem}
    Let the complex coefficient $A$ allow the DG-N-M properties \rref{boundednesselliptic} and \rref{holderelliptic} to hold for any weak solution to $ -\DIV(A\nabla w )=0$ and $ -\DIV(A^\ast\nabla w )=0$ in domain $\Omega\in\rn\times\reals$, then for any cube $Q_r=Q_r(X_0)$ whose concentric cube $Q_{\sigma r}=Q_{\sigma r}(X_0)\subset\Omega$ with $\sigma>1$, for any weak solution to $\mathcal{L}u=0$ we have
    \begin{align}
        \sup_{Q_r} |u| &\le {C}_{\sigma} \brackets{\fint_{ Q_{\sigma r}} |u|^2}^\frac{1}{2}\brackets{1+r^2\fint _{Q_{\sigma r}}Vd\xi}^{N_0}\label{degiorginash for schrodinger}\\
         |u(X)-u(Y)| &\le C_\sigma\frac{|X-Y|^\alpha}{r^\alpha} \brackets{\fint_{Q_{\sigma r}} |u|^2}^\frac{1}{2}\brackets{1+r^2\fint _{Q_{\sigma r}}Vd\xi}^{N_0+1}\label{moser for schrodinger}
    \end{align}
    for some constant $N_0>0$ and any $X, Y\in Q_r$, where $\alpha=\min\{\delta_0,\alpha_0\}$. 
\end{theorem} 
\begin{proof}
    We follow the method of \cite[Lemma 3.3]{GlobalWeightedHigherOrder}. Set a point $X_0=(x_0,t_0)$ and briefly write $Q_r(X_0)=Q_r$. Take a bump function $\psi\in C^\infty_c(\rnp) $, $\supp\psi\subset Q_{(3+ \sigma )r/4 }(x_0,t_0)$ and $\psi|_{Q_{(7+{\sigma})r/{8}}}=1$ with $|\nabla\psi|\le \frac{C}{(\sigma-1)r}$, concerning $u\psi =L^{-1}_0(-\DIV\nabla (u\psi))$, we have
    \begin{align*}
        &\left| u(X)\psi(X)\rabs\\=&\labs\int_{\rnp_+}\left(A(\xi) \nabla_\xi\Gamma_0(X |\xi,\tau) \nabla\psi(\xi,\tau)\cdot u(\xi,\tau)\right.\right.\\&\left.\left.-\Gamma_0(X|\xi,\tau)\brackets{A\nabla u \nabla\psi+V(\xi)u\psi}\right) d\xi d\tau\right|\\
        \le & \frac{C}{ r}\int_{Q_{(3+ \sigma )r/4}\setminus Q_{(7+ \sigma )r/8}}|\nabla_\xi\Gamma_0(X|\xi,\tau) ||u(\xi,\tau)|d\xi d\tau\\&+\frac{C}{\sigma r}\int_{Q_{(3+ \sigma )r/4}\setminus Q_{(7+ \sigma )r/8}}|\Gamma_0(X|\xi,\tau)||\nabla u|d\xi d\tau\\&+C\int_{Q_{(3+ \sigma )r/4}}|\Gamma_0(X|\xi,\tau)||u(\xi,\tau)| V(\xi) d\xi d\tau\\
        \defeq&U_1(X)+U_2(X)+U_3(X).
    \end{align*}

    By Caccioppoli's inequality, we have
    \begin{align}
        U_1(X)\le& \frac{C}{ r^2}\brackets{\int_{Q_{(1+ \sigma )r/2}\setminus Q_{(15+ \sigma )r/16}}| \Gamma_0(X|\xi,\tau)|^2 d\xi d\tau}^\frac{1}{2}\brackets{\int_{Q_{(3+ \sigma )r/4}}|u|^2}^{\frac{1}{2}}\nonumber\\
        \le & C\brackets{\fint_{Q_{(3+ \sigma )r/4}}|u|^2}^\frac{1}{2}.\label{u1bounded}
    \end{align}
    
    Similarly, we have
    \begin{align}\label{u2bounded}
        U_2(X)\le  C\brackets{\fint_{Q_{(3+ \sigma )r/4}}|u|^2}^\frac{1}{2}.
    \end{align}

    To handle the term $U_3$, first we notice that when $u\in L^p_{\textup{loc}}$ where $p>\frac{q(n+1)}{2q-n-1}$, by H\"older inequality $Vu\in L^{p_1}_{\textup{loc}}$ where $\frac{1}{p_1}=\frac{1}{p}+\frac{1}{q}<\frac{2}{n+1}$. Thus we have
    \begin{align*}
        U_3(X)\le& C r^{n+1}\brackets{\fint_{Q_{(3+\sigma)r/4}} (V|u|)^{p_1}}^\frac{1}{p_1}\brackets{\fint_{Q_{2 r}}\frac{d\xi d\tau}{(|\xi|+|\tau|)^{(n-1)\frac{p_1}{p_1-1}}}}^{1-\frac{1}{p_1}} \\
        \le& Cr^2\brackets{\fint _{Q_{(3+ \sigma )r/4}(x_0)}Vd\xi}\brackets{\fint_{Q_{(3+ \sigma )r/4}}|u|^p}^\frac{1}{p}. 
    \end{align*}

    If $p<\frac{q(n+1)}{2q-n-1}$, by fractional integral formula we have

    \begin{align}
        \brackets{\fint_{Q_{(3+\sigma)r/4}}U_3(X)^{p_2} dX}^\frac{1}{p_2}\le &Cr^{-{\frac{n+1}{p_2}}+\frac{n+1}{p_1}}\brackets{\fint_{Q_{(3+\sigma)r/4}}(V|u|)^{p_1}}^\frac{1}{p_1} \nonumber\\\le & Cr^2\fint_{Q_{(3+\sigma)r/4}(x_0)}Vd\xi\brackets{\fint_{Q_{(3+\sigma)r/4}} |u|^p}^\frac{1}{p},\label{u3iteration}
    \end{align}

    where $\frac{1}{p_2}=\frac{1}{p_1}-\frac{2}{n+1}=\frac{1}{p}+\frac{1}{q}-\frac{2}{n+1}<\frac{1}{p}$.

    Thus we can iterate \rref{u3iteration} together with \rref{u1bounded} and \rref{u2bounded} finite times, elevating $p$ from $2$ to $\infty$ to get \rref{degiorginash for schrodinger}, provided that $\frac{q(n+1)}{4q-2n-2} $ is not an integer. Otherwise we can also adjust $q$ slightly to reach $p=\infty$, thus \rref{degiorginash for schrodinger} holds. 
    
    For the H\"older continuity of $u$, by similar method we have

    \begin{align*}
        &\labs u(X)\psi(X)-u(X_0)\psi(X_0)\rabs\\
        \le & \frac{C}{ r}\int_{Q_{(3+ \sigma )r/4}\setminus Q_{(7+ \sigma )r/8}}\left|\nabla_\xi\left(\Gamma_0(X|\xi,\tau)-\Gamma_0(X_0|\xi,\tau)\right)\right||u(\xi,\tau)|d\xi d\tau\\& +\frac{C}{ r}\int_{Q_{(3+ \sigma )r/4}\setminus Q_{(7+ \sigma )r/8}}|\Gamma_0(X|\xi,\tau)-\Gamma_0(X_0|\xi,\tau)||\nabla u|d\xi d\tau\\&+C\int_{Q_{(3+ \sigma )r/4}}|\Gamma_0(X|\xi,\tau)-\Gamma_0(X_0|\xi,\tau)||u(\xi,\tau)| V(\xi) d\xi d\tau\\
        \defeq&W_1(X)+W_2(X)+W_3(X).
    \end{align*}

    By Caccioppoli's inequality and \rref{holderelliptic} we know that
    \begin{align*}
        W_j(X)\le& \frac{C}{ r^2}\brackets{\fint_{Q_{(3+ \sigma )r/4}\setminus Q_{(7+ \sigma )r/8}}\labs \Gamma_0(X|\xi,\tau)-\Gamma_0(X_0|\xi,\tau)\rabs^2 d\xi d\tau}^\frac{1}{2}\brackets{\fint_{Q_{(3+ \sigma )r/4}} |u|^2}^\frac{1}{2}\\
        \le & C\brackets{\frac{|X-X_0|}{r}}^{\alpha_0} \brackets{\fint_{Q_{(3+ \sigma )r/4}} |u|^2}^\frac{1}{2}
    \end{align*}
    for $j=1$ or $2$. 

    The trickiest part is $W_3$ and we divide it into two parts. Let $h=|X-X_0|<r$, we write

    \begin{align*}
        W_3(X)\le& C\left(\int_{Q_{(1+{\sigma})h/2}}+\int_{Q_{(1+\sigma)r/2}\setminus Q_{(1+\sigma)h/2}}\right)|\Gamma_0(X|\xi,\tau)-\Gamma_0(X_0|\xi,\tau)||u(\xi,\tau)| V(\xi) d\xi d\tau \\
        \defeq&W_4(X)+W_5(X).
    \end{align*}
        
    Choose $\alpha_1=\min\left\{\alpha_0,\delta_0\right\}$ so that $V\in \mathcal{B}_{\frac{n}{2-\alpha_1}}$. Then by the H\"older continuity of $\Gamma_0$, the term $W_5(X)$ can be controlled by
    \begin{align*}
        W_5(X)\le &Ch^{\alpha_1}\int_{Q_{3r}}\frac{V(\xi)}{(|\xi-x_0|+|\tau-t_0|)^{n-1+\alpha_1}}d\xi d\tau\cdot\sup_{Q_{(1+{\sigma})r/2}}|u|\\
        \le &C\brackets{\frac{h}{r}}^{\alpha_1} r^2 \fint_{Q_{3r}}Vd\xi \brackets{\fint_{Q_{\sigma r}}|u|^2}^\frac{1}{2},
    \end{align*}
    where we have used the previous result \rref{degiorginash for schrodinger}.
    
    Finally, the term $W_4$ can be converted into
    \begin{align*}
        W_4(X)\le& C\int_{Q_{3h}}\frac{V(\xi)
        }{(|\xi-x_0|+|\tau-t_0|)^{n-1}}d\xi d\tau\cdot\sup_{Q_{\brackets{(1+{\sigma})}r/{2}}}|u|\\
        \le &C\brackets{\frac{h}{r}}^{\delta_0}r^2\fint_{Q_{3r}}V(\xi)d\xi\brackets{\fint_{Q_{\sigma r}}|u|^2}^\frac{1}{2},
    \end{align*}
    hence the proof of \rref{moser for schrodinger} finishes.
\end{proof}

{\bf Acknowledgment:}\quad 
The  research was supported  by the NNSF
(12471089)  of China.\\

{\bf Declaration:}\quad  On behalf of all authors, the corresponding author states that there is no conflict of interest.

\bibliographystyle{plain}
\bibliography{schrodingerbib}

\medskip

LMAM, School of Mathematical Sciences,
 Peking University, Beijing, 100871,
 P. R. China

 Lin Tang,\quad
 E-mail address:  tanglin@math.pku.edu.cn

Botian Xiao,\quad
 E-mail address:  xbt1828@163.com

\end{document}